\documentclass[a4paper,10pt]{amsart}

\title{On the average condition number of tensor rank decompositions}
\author{Paul Breiding}
\thanks{PB: Max Planck Institute for Mathematics in the Sciences,
Inselstrasse 22--26, 04103 Leipzig, Germany. Email: breiding@mis.mpg.de. Partially supported by DFG research grant BU 1371/2-2.}
\author{Nick Vannieuwenhoven}
\thanks{NV: KU Leuven, Department of Computer Science,
Celestijnenlaan 200A, 3001 Heverlee, Belgium. Email: nick.vannieuwenhoven@kuleuven.be. Supported by a Postdoctoral Fellowship of the Research Foundation--Flanders.}
\newcommand{\thelanguage}{english}

\usepackage{geometry}
\usepackage{mathrsfs}
\usepackage[\thelanguage]{babel}
\usepackage{tkz-euclide, tikz, pgfplots}
\usetikzlibrary{backgrounds,patterns,matrix,calc,arrows,positioning,fit,shapes.geometric,decorations.pathmorphing}
\usepackage{tikz-cd}
\usepackage{float}
\usepackage{amsthm,amsmath,amssymb,amsfonts}
\usepackage{cleveref}
\usepackage{multirow}
\usepackage{enumitem}
\usepackage{tabularx}
\usepackage{booktabs}
\usepackage[sort]{cite}

\usepackage{marginnote}
\usepackage{etex}

\newcommand{\set}[1]{\left\{#1\right\}}
\newcommand{\cset}[2]{\left\{#1\mid #2\right\}}

\newcommand{\Norm}[1]{\lVert #1 \rVert}

\renewcommand{\d}{\mathrm{d}}
\DeclareMathOperator*{\Prob}{\mathrm{Prob}}
\DeclareMathOperator*{\mean}{\mathbb{E}}

\newcommand{\tuple}[1]{\mathfrak{#1}}

\newcommand{\Var}[1]{\mathcal{#1}}

\newcommand{\tensor}[1]{\mathfrak{#1}}
\newcommand{\vect}[1]{\mathbf{#1}}
\newcommand{\sten}[3]{\vect{#1}_{#2}^{#3}}
\newcommand{\Tang}[2]{\mathrm{T}_{#1} {#2}}

\newcommand{\R}{\mathbb{R}}
\newcommand{\deriv}[2]{\mathrm{d}_{#2} #1}

\newcommand{\Gr}{\mathrm{Gr}}

\newcommand{\GrSigma}{\Sigma_{\Gr}}

\newcommand{\dist}{\mathrm{dist}}

\numberwithin{equation}{section}
\numberwithin{figure}{section}
\numberwithin{table}{section}

\theoremstyle{plain}
\newcounter{numbering} \numberwithin{numbering}{section}
\newtheorem{thm}[numbering]{Theorem}
\newtheorem{lemma}[numbering]{Lemma}
\newtheorem{prop}[numbering]{Proposition}
\newtheorem{cor}[numbering]{Corollary}
\theoremstyle{definition}

\newtheorem{conjecture}[numbering]{Conjecture}
\newtheorem{dfn}[numbering]{Definition}
\theoremstyle{remark}

\newtheorem{rem}{Remark}

\crefname{equation}{}{}
\crefname{equation}{}{}
\crefname{figure}{Figure}{Figures}
\crefname{section}{Section}{Sections}
\crefname{table}{Table}{Tables}
\crefname{lemma}{Lemma}{Lemmata}
\crefname{prop}{Proposition}{Propositions}
\crefname{thm}{Theorem}{Theorems}
\crefname{cor}{Corollary}{Corollaries}
\crefname{dfn}{Definition}{Definitions}
\crefname{notation}{Notations}{Notations}
\crefname{rem}{Remark}{Remarks}
\crefname{claim}{Claim}{claims}
\newcommand{\Pj}{\mathbb{P}}

\newcommand{\refthm}[1]{{\cref{#1}}}
\newcommand{\reflem}[1]{{\cref{#1}}}
\newcommand{\refeqn}[1]{{\cref{#1}}}

\begin{document}
\begin{abstract}
  We compute the expected value of powers of the geometric condition number of random tensor rank decompositions. It is shown in particular that the expected value of the condition number of $n_1\times n_2 \times 2$ tensors with a random rank-$r$ decomposition, given by factor matrices with independent and identically distributed standard normal entries, is infinite. This entails that it is expected and probable that such a rank-$r$ decomposition is sensitive to perturbations of the tensor. Moreover, it provides concrete further evidence that tensor decomposition can be a challenging problem, also from the numerical point of view. On the other hand, we provide strong theoretical and empirical evidence that tensors of size $n_1~\times~n_2~\times~n_3$ with all $n_1,n_2,n_3 \ge 3$ have a finite average condition number. This suggests there exists a gap in the expected sensitivity of tensors between those of format $n_1\times n_2 \times 2$ and other order-3 tensors. For establishing these results, we show that a natural weighted distance from a tensor rank decomposition to the locus of ill-posed decompositions with an infinite geometric condition number is bounded from below by the inverse of this condition number. That is, we prove one inequality towards a so-called condition number theorem for the tensor rank decomposition.
\end{abstract}
\maketitle

{\small {\bf Keywords}
tensor rank decomposition; CPD; condition number; ill-posed problems; inverse distance to ill-posedness; average complexity;
}

{\small {\bf Subject class} Primary 49Q12, 53B20, 15A69; Secondary 14P10, 65F35, 14Q20
49Q12  	Sensitivity analysis
53B20  	Local Riemannian geometry
15A69  	Multilinear algebra, tensor products
14P10  	Semialgebraic sets and related spaces
65F35  	Matrix norms, conditioning, scaling
14Q20  	Effectivity, complexity  (computation in AG)
}

\section{Introduction}
Whenever data depends on several variables, it may be stored as a $d$-array
$$
 \tensor{A} = \begin{bmatrix} a_{i_1,i_2,\ldots,i_d} \end{bmatrix}_{i_1,i_2,\ldots,i_d=1}^{n_1,n_2,\ldots,n_d} \in \R^{n_1 \times n_2 \times \cdots \times n_d}.
$$
For the purpose of our exposition, this $d$-array is informally called a \emph{tensor}.
Due to the curse of dimensionality, plainly storing this data in a tensor is neither feasible nor insightful. Fortunately, the data of interest often admit additional structure that can be exploited. One particular tensor decomposition is the \textit{tensor rank decomposition}, or \textit{canonical polyadic decomposition} (CPD). It was proposed by \cite{hitchcock} and expresses a tensor $\tensor{A} \in \R^{n_1 \times n_2 \times \cdots \times n_d}$ as a minimum-length linear combination of rank-$1$ tensors:
\begin{equation}\tag{CPD}\label{CPD}
 \tensor{A} =  \tensor{A}_1+\tensor{A}_2 + \cdots+\tensor{A}_r,\quad\text{where}\quad \tensor{A}_i = \sten{a}{i}{1} \otimes \sten{a}{i}{2} \otimes \cdots \otimes \sten{a}{i}{d},
\end{equation}
and where $\otimes$ is the \textit{tensor product}:
\begin{equation}\label{tensor_in_coordinates}
 \sten{a}{}{1} \otimes \sten{a}{}{2} \otimes \cdots \otimes \sten{a}{}{d} =
 \begin{bmatrix}
 a_{i_1}^{(1)} a_{i_2}^{(2)} \cdots a_{i_d}^{(d)}
\end{bmatrix}_{i_1,i_2,\ldots,i_d=1}^{n_1,n_2,\ldots,n_d} \in \R^{n_1 \times n_2 \times \cdots \times n_d}, \quad \text{where $\sten{a}{}{k} = [a_{i}^{(k)}]_{i=1}^{n_k}$.}
\end{equation}
The smallest $r$ for which the expression \refeqn{CPD} is possible is called the \emph{rank} of $\tensor{A}$. In several applications, the CPD of a tensor reveals domain-specific information that is of interest, such as in psychometrics \cite{Kroonenberg2008}, chemical sciences \cite{SBG2004}, theoretical computer science \cite{BCS1997}, signal processing \cite{Comon1994,CJ2010,Review2016}, statistics \cite{AMR2009,M1987} and machine learning \cite{AGHKT2014}. In most of these applications, the data that the tensor represents is corrupted by measurement errors, which will cause the CPD computed from the measured data to differ from the CPD of the true, uncorrupted data.

For measuring the sensitivity of a computational problem to perturbations in the data, a standard technique in numerical analysis is investigating the \emph{condition number} \cite{condition,higham}. Earlier theoretical work by the authors introduced two related condition numbers for the computational problem of computing a CPD from a given tensor; see \cite{V2017,BV2017}. Let us recall the definition of the geometric condition number of the tensor rank decomposition of \cite{BV2017}. The set of rank-1 tensors $\Var{S}\subset \R^{n_1\times \cdots \times n_d}$ is a smooth manifold, called \emph{Segre manifold}. The set of tensors of rank at most~$r$ is given as the image of the addition map
$\Phi:\Var{S}^{\times r}\to \R^{n_1\times \cdots \times n_d}, (\tensor A_1,\ldots, \tensor A_r)\to \tensor A_1 + \cdots + \tensor A_r$. The condition number of $\tensor{A}$ is defined locally\footnote{Consult \cite[Section 1]{BV2017} for an explanation why a local definition is required.} at the decomposition $(\tensor A_1,\ldots, \tensor A_r)$ as
  $$
  \kappa(\tensor{A},(\tensor A_1,\ldots, \tensor A_r)) := \lim\limits_{\epsilon \to 0} \,\sup\limits_{\tensor{B} \text{ has rank } r, \atop \Vert \tensor{A} - \tensor{B}\Vert < \epsilon} \,\frac{\Vert\Phi^{-1}(\tensor{A}) - \Phi^{-1}(\tensor{B})\Vert}{\Vert \tensor{A} - \tensor{B}\Vert},
  $$
where $\Phi^{-1}$ is the local inverse of $\Phi$ with $\Phi^{-1}(\tensor{A}) ={(\tensor A_1,\ldots, \tensor A_r)}$. If such a local inverse does not exist, we define $\kappa(\tensor{A},(\tensor A_1,\ldots, \tensor A_r)):=+\infty$. The norms are the Euclidean norms induced by the ambient spaces of the domain and image of $\Phi$. As $\tensor{A}$ depends uniquely on $(\tensor A_1,\ldots, \tensor A_r)$ we write~$\kappa(\tensor A_1,\ldots, \tensor A_r)$ for the condition number.

The topic of this paper is the first inquiry into a probabilistic analysis of the condition number of the CPD; see, e.g., \cite{condition,Cucker16}. In particular, we focus on the average analysis and compute the expected value of powers of the condition number for \emph{random rank-$1$ tuples} $(\lambda_1 \tensor{A}_1, \ldots, \lambda_r \tensor{A}_r)$ of length $r$, where the $\lambda_i \in \R\setminus\{0\}$ are arbitrary and $\tensor{A}_i := \sten{a}{i}{1}\otimes \cdots \otimes \sten{a}{i}{d}$ in which the~$\sten{a}{i}{j}~\in~\R^{n_j}$ have independently and identically distributed (i.i.d.) standard normal entries.
This distribution is very relevant for scientific research, as samples from it are often employed to test the effectiveness of algorithms for computing CPDs. In \cite[Proposition 7.1]{BV2017} we have shown that the condition number is invariant under scaling of the rank-one tensors $\tensor{A}_i$. For this reason, we assume, without loss of generality, that $\lambda_1 = \cdots = \lambda_r = 1$ in the remainder of this paper. One of the main results we will prove is the following statement.

\begin{cor} \label{cor_expected_value}
 Let $(\tensor{A}_1, \ldots, \tensor{A}_r) \in \Var{S}^{\times r}$ be a random rank-$1$ tuple in $\R^{n_1 \times n_2 \times n_3}$, where $n_1 \ge n_2 \ge n_3 \ge 2$ and $r \ge 2$. Then, we have
$\mathbb{E}\, \bigl[ \kappa(\tensor{A}_1,\ldots,\tensor{A}_r)^{c} \bigr] = \infty$, for all $c \ge n_3-1$.
\end{cor}

In particular, the corollary implies that the expected value of the condition number---without a power---of random rank-$1$ tuples in $\R^{n_1 \times n_2 \times 2}$ is $\infty$. This result provides further concrete evidence that the problem of computing a CPD can have a high condition number with a nonnegligible probability. See, for example, the curve $n=2$ in \cref{figure_powerdistribution} which shows the complementary cumulative distribution function of the condition number of random rank-$1$ tuples of length $7$ in $\R^{7 \times 7 \times 2}$. It shows that there is a $10\%$ chance the condition number is greater than $10^4$, and a $1\%$ chance that it is greater than $4 \cdot 10^5$. In many applications where the CPD is employed, the measurement errors are not sufficiently small to compensate such high condition numbers.

\Cref{cor_expected_value} is a contribution to a body of research illustrating that computing CPDs can be a very challenging problem. The result of \cite{Hastad1990} is often cited in this regard. H\r{a}stad reduces 3SAT to computing the rank of a tensor, which shows that the latter problem is NP-complete in the Turing machine computational model. However, this does not entail that computing a \emph{typical} CPD is a difficult problem. Another oft-cited result by \cite{dSL2008} relates to the difficulty of approximating CPDs; they proved that the problem of computing the best rank-$2$ approximation is ill-posed on an open set in $\R^{n_1 \times n_2 \times n_3}$. Further evidence originates from the sensitivity to perturbations of the CPD: \cite{V2017} illustrated numerically that the norm-balanced condition number can blow up near the ill-posed locus of \cite{dSL2008}; subsequently \cite{BV2017} proved that the geometric condition number will diverge to infinity when approaching the ill-posed locus. Recall from \cite[Theorem 1]{BV2018} that the condition number appears in estimates of the rate of convergence and radii of attraction of Riemannian Gauss--Newton methods for computing a best rank-$r$ approximation of a tensor, such as the ones in \cite{BV2018,BV2018b}. \Cref{cor_expected_value} thus not only shows that computing CPDs is a difficult problem, but also reinforces the result about the high computational complexity of computing low-rank approximations. Nevertheless, the present article is the first to study average complexity.

There are two new key insights that this paper offers. The first is decidedly negative: the average condition number of random rank-$1$ tuples of length $r$ in $\R^{n_1 \times n_2 \times 2}$ is infinite, implying that it is \textit{probable} to sample a CPD with a high condition number; see \cref{sec_numex_distribution}. However, the second one is considerably more positive: our inability to reduce the value of $c$ in \cref{cor_expected_value} to $c=1$, or even any value less than $n_3-1$, in our analysis, should, in combination with the empirical evidence in \cref{sec_numex_distribution} and the impossibility result in \cref{prop_impossible}, be taken as clear evidence for the following conjecture.

\begin{conjecture}
There exists an integer $2 \le r^\star \le \frac{n_1 n_2 n_3}{n_1+n+2+n_3-2}$ such that for all $1\leq r\leq r^\star$ and  $n_1 \ge n_2 \ge n_3 \ge 3$ the expected condition number of random rank-$1$ tuples of length $r$ in~$\R^{n_1 \times n_2 \times n_3}$ is finite.
\end{conjecture}

This would suggest there exists a gap in sensitivity (which is one measure of complexity, as explained above) between $n_1 \times n_2 \times 2$ tensors or pairs of $n_1 \times n_2$ matrices, where the average condition number is proved to be $\infty$, and more general $n_1 \times n_2 \times n_3$ tensors with $n_1, n_2, n_3 > 2$, where all empirical and theoretical evidence points to a finite average condition number. This is similar to the gap in classic complexity between order-$2$ tensors and order-$d$ tensors with $d \ge 3$ for computing the tensor rank. It is noteworthy that increasing the size of the tensor seems to decrease the complexity of computing the CPD.

\subsection*{Statement of the technical contributions}
We proved in \cite[Theorem 1.3]{BV2017} that the condition number of the CPD is equal to the distance to ill-posedness in an \emph{auxiliary space}: according to the theorem the condition number of the CPD $\kappa(\tensor A_1,\ldots,\tensor A_r)$ at a decomposition $(\tensor A_1,\ldots,\tensor A_r)\in\Var S^{\times r}$ is equal to the inverse distance of the tuple of tangent spaces $(\Tang{\tensor A_1}{\Var{S}},\ldots, \Tang{\tensor A_r}{\Var{S}})$ to ill-posedness:
\begin{equation}\label{characterization123}
\kappa(\tensor A_1,\ldots,\tensor A_r) = \frac{1}{\dist_\mathrm{P}((\Tang{\tensor A_1}{\Var{S}},\ldots, \Tang{\tensor A_r}{\Var{S}}) ,\Sigma_\Gr)},
\end{equation}
where $\Sigma_\Gr$ and the distance $\dist_\mathrm{P}$ are defined as follows.
Let $n:=\dim \Var S$ and write $\Pi:=n_1\cdots n_d$ for the dimension of $\R^{n_1\times \cdots \times n_d}$. Denote by $\Gr(\Pi,n)$ the \emph{Grassmann manifold} of $n$-dimensional linear spaces in the space of tensors $\R^{n_1\times \cdots \times n_d}\cong \R^\Pi$. Then, the tuple of tangent spaces to $\Var S$ at the decomposition $(\tensor A_1,\ldots,\tensor A_r)$ is an element in the product of Grassmannians: $(\Tang{\tensor A_1}{\Var{S}},\ldots, \Tang{\tensor A_r}{\Var{S}})\in \Gr(\Pi,n)^{\times r}$. The set~$\Sigma_\Gr$ in \refeqn{characterization123} is then defined as the $r$-tuples of linear spaces that are not in general position. In formulas:
\begin{equation}\label{GrSigma}
\GrSigma:= \cset{(W_1,\ldots,W_r)\in \Gr(\Pi,n)^{\times r}}{\dim (W_1+\cdots+ W_r) < rn }.
\end{equation}
The distance measure in \refeqn{characterization123} is the \emph{projection distance} on $\Gr(\Pi,n)$. It is defined as $\Vert \mathrm{pr}_{V}-\mathrm{pr}_{W}\Vert$, where $\mathrm{pr}_{V}$ and $\mathrm{pr}_{W}$ are the orthogonal projections on the spaces $V$ and~$W$ respectively, and $\Vert \cdot\Vert$ is the spectral norm. This distance is extended to $\Gr(\Pi,n)^{\times r}$ in the usual~way:
 \begin{equation}\label{projection_distance}\dist_{\mathrm{P}}((V_1,\ldots,V_r), (W_1,\ldots,W_r)):=\sqrt{\sum_{i=1}^r \Vert \pi_{V_i}-\pi_{W_i}\Vert^2}.\end{equation}

The decomposition $(\tensor{A}_1,\ldots,\tensor{A}_r)$ whose corresponding tangent space lies in $\Sigma_\Gr$ is \textit{ill-posed} in the following sense. It was shown in \cite[Corollary 1.2]{BV2017} that whenever there is a smooth curve~\(
\gamma(t) = (\tensor{A}_1(t), \ldots, \tensor{A}_r(t))
\)
such that
\(\tensor{A} = \sum_{i=1}^r \tensor{A}_i(t)
\)
is constant, even though $\gamma'(0)~\ne~0$, then all of the decompositions $(\tensor{A}_1(t), \ldots, \tensor{A}_r(t))$ of $\tensor{A}$ are ill-posed decompositions. Note that in this case, the tensor $\tensor{A}$ thus has a family of decompositions running through $(\tensor{A}_1(0), \ldots, \tensor{A}_r(0))$. We say that $\tensor{A}$ is not \emph{locally $r$-identifiable}.  Tensors are expected to admit only a finite number of decompositions, generically (for the precise statements see, e.g., \cite{AOP2009,CO2012,BCO2014,COV2014}). Therefore, tensors that are not locally $r$-identifiable are very special as their parameters cannot be identified uniquely. Ill-posed decompositions are exactly those that, \emph{using only first-order information, are indistinguishable from decompositions that are not locally $r$-identifiable}.

In this article,
we relate the condition number to a metric on the data space $\Var S^{\times r}$; see \cref{cnt}. Following \cite{Demmel1987}, we then use this result and show in \cref{thm_expected_value} that the expected value of the condition number is infinite whenever the ill-posed locus in $\Var S^{\times r}$ is of codimension 1.
To describe the condition number as an inverse distance to ill-posedness on $\Var{S}^{\times r}$ we need to consider an angular distance. This is why the main theorem of this article, \cref{cnt}, is naturally stated in projective space.

\begin{thm}\label{cnt}
Denote by $\pi: \R^{n_1\times\cdots\times n_d}\backslash\{0\} \to \mathbb{P}(\R^{n_1\times\cdots\times n_d})$ the canonical projection onto projective space. We put $\mathbb{P}\Var{S}:=\pi(\Var S)$ and for tensors $\tensor A\in\R^{n_1\times\cdots\times n_d}$ we denote the corresponding class in projective space by $[\tensor A]:=\pi(\tensor A)$. Let $(\tensor A_1,\ldots,\tensor A_r)\in\Var{S}^{\times r}$. Then,
$$
\kappa(\tensor A_1,\ldots,\tensor A_r)
\geq \frac{1}{\dist_{\mathrm{w}}(([\tensor A_1],\ldots,[\tensor A_r]), \Sigma_\mathbb{P})},
$$
where
$$
\Sigma_\mathbb{P}=\cset{([\tensor A_1],\ldots,[\tensor A_r])\in (\mathbb{P}\Var{S})^{\times r}}{\kappa(\tensor A_1,\ldots, \tensor A_r)=\infty}
$$
and the distance $\dist_{\mathrm{w}}$ is defined in \cref{weighted_distance}.
\end{thm}

This characterization of a condition number as an inverse distance to ill-posedness is a called \emph{condition number theorem} in the literature and it provides a geometric interpretation of complexity of a computational problem. \cite{Demmel1987} advocates this characterization as it may be used to ``compute the probability distribution of the distance from a `random' problem to the set [of ill-posedness].'' Condition number theorems were, for instance, derived for matrix inversion \cite{kahan,eckart_young,demmel2}, polynomial zero finding \cite{Hough,demmel2}, and computing eigenvalues \cite{WILKINSON,demmel2}.
For a comprehensive overview see \cite[pages~10,~16,~125,~204]{condition}.
We use the above condition number theorem to derive a result on the average condition number of CPDs.

\begin{thm}\label{thm_expected_value}
Let $(\tensor{A}_1,\ldots,\tensor{A}_r)\in \Var S^{\times r}$, $r \ge 2$, be a random rank-1 tuple in $\R^{n_1\times\cdots\times n_d}$. Let $e \ge c \ge 1$. If~$\Sigma_\Pj$ contains a manifold of codimension $0$ or $c$ in $\Var S^{\times r}$, then
$\mathbb{E}\,\bigl[ \kappa(\tensor{A}_1,\ldots,\tensor{A}_r)^{e} \bigr] = \infty.$
\end{thm}

In Section~\ref{sec_expected_value}, we prove that for the format $n_1 \times n_2 \times n_3$, $n_1 \ge n_2 \ge n_3 \ge 2$, the ill-posed locus $\Sigma_\mathbb{P}$ contains a submanifold that is of codimension $n_3-1$ in $\Var S^{\times r}$. Hence, the aforementioned \cref{cor_expected_value} is obtained as a consequence of \cref{thm_expected_value}.

\begin{rem}
The statement of \cref{cor_expected_value} can easily be strengthened as follows. It is known from dimensionality arguments about fibers of projections of projective varieties that there exists an integer critical value $r^\star \le \frac{\dim \R^{n_1 \times \cdots \times n_d}}{\dim \Var{S}}$ such that every tensor of rank $r > r^\star$ has at least a $1$-dimensional variety of rank decompositions in $\Var{S}^{\times r}$; see, e.g., \cite{AOP2009,Harris1992,Landsberg2012}. Specifically,~$r^\star$ is the smallest value such that the dimension of the projective $(r^\star+1)$-secant variety of $\mathbb{P}(\Var{S})$ is strictly less than $(r^\star + 1) \dim \Var{S} - 1$. It follows then from \cite[Corollary 1.2]{BV2017} that the condition number $\kappa(\tensor{A}_1,\ldots,\tensor{A}_r)=\infty$ for \emph{all} decompositions $(\tensor{A}_1,\ldots,\tensor{A}_r)$ when $r > r^\star$. For smaller values of $r$, we can only prove the statement in \cref{cor_expected_value}.
\end{rem}

\subsection*{Structure of the article}
The rest of this paper is structured as follows. In the next section, we recall some preliminary material on Riemannian geometry. We start by proving the main contribution in \cref{sec_expected_value}, namely \cref{thm_expected_value}, because its proof is less technical. \Cref{sec:proof} is devoted to the proof of the condition number theorem, namely \cref{cnt}. In \cref{sec:experiments}, we present some numerical experiments and computer algebra computations illustrating the main contributions. Finally, the paper is concluded in \cref{sec_conclusions}.

\subsection*{Acknowledgements} We thank C. Beltr\'an for pointing out \cref{prop:inversedistance} to us, so that we could use \cref{cnt} to obtain \cref{thm_expected_value}. We like to thank P. B\"urgisser for carefully reading through the proof of \cref{isometry_thn}. Anna Seigal is thanked for discussions relating to \cref{lem_codim_one}, which she discovered independently. Some parts of this work are also part of the PhD thesis \cite{Breiding2017} of the first author.

\section{Preliminaries and notation} \label{sec_preliminaries}
We denote the standard Euclidean inner product on $\mathbb{R}^m$ by $\langle\cdot,\cdot\rangle$.
The real projective space of dimension $m-1$ is denoted by $\Pj(\R^m)$ and the unit sphere of dimension~$m-1$ is denoted by $\mathbb{S}(\R^m)$. Points in linear spaces are typeset in bold-face lower-case symbols like $\vect{a},\vect{x}$. Points in projective space or other manifolds are typeset in lower-case letters like $a,x$. The \emph{orthogonal complement} of a point $\vect{x}\in \R^m$ is~$\vect{x}^\perp:=\{\vect{y}\in\R^m \mid \langle \vect{x},\vect{y}\rangle =0\}$.
We write $\Var S$ for the Segre manifold in $\R^{n_1\times \cdots \times n_d}$. If it is necessary to clarify the parameters, we also write $\Var S_{n_1,\ldots,n_d}$ .
Throughout this paper, $n$ denotes the dimension of $\Var S$:
\begin{equation}\label{dim_segre}
  n := \dim \Var S_{n_1,\ldots,n_d} = 1-d+\sum_{i=1}^d n_i;
  \end{equation}
see \cite{Harris1992,Landsberg2012}.
The projective Segre map is
\begin{align} \label{eqn_def_sigma}
  \sigma : \mathbb{P}(\R^{n_1})\times \cdots \times \mathbb{P}(\R^{n_d}) \to \mathbb{P}\Var{S},\; ([\sten{a}{}{1}],\ldots,[\sten{a}{}{d}])\mapsto[\sten{a}{}{1}\otimes \cdots \otimes \sten{a}{}{d}];
  \end{align}
see \cite[Section 4.3.4.]{Landsberg2012}.

Let $(M,g)$ be a Riemannian manifold. For $x\in M$ we write $\Tang{x}{M}$ for the tangent space of $M$ at~$x$. For~$\gamma : (-1,1) \to M$ a smooth curve in $M$ we will use the shorthand notations $\gamma'(0) := \frac{\mathrm{d}}{\mathrm{d}t}|_{t=0} \gamma(t)$ for the tangent vector in~$\Tang{\gamma(0)}{M}$ and $\gamma'(t) := \frac{\mathrm{d}}{\mathrm{d}t} \gamma(t)$. Recall that
the \emph{Riemannian distance} between two points $p,q\in M$ is $\dist_M(p,q)= \inf\cset{l(\gamma)}{\gamma(0)=p,\gamma(1)=q}$. The infimum is over all piecewise differentiable curves $\gamma:[0,1]\to M$ and the length of a curve is
$
l(\gamma)=\int_0^1 g(\gamma'(t), \gamma'(t))^\frac{1}{2} \,\d t
$.
The distance $\dist_M$ makes~$M$ a metric space \cite[Proposition 2.5]{riemannian_geometry}.

We use the symbol $\vert \omega\vert$ to denote the \emph{density} on $M$ given by $g$ \cite[Proposition 16.45]{Lee2013}.
For densities with finite volume, i.e., $\int_M \vert\omega\vert <\infty$, this defines the \emph{uniform distribution}:
$$\Prob_{X \text{ uniformly in } M}\{X\in N\} := \frac{1}{\int_M \vert\omega\vert} \,\int_N \vert\omega\vert\quad \text{ where } N\subset M.$$

A particularly important manifold in the context of this article is the projective space $\Pj(\R^m)$. An atlas for $\Pj(\R^m)$ is, for instance, given by the affine charts $(U_i,\varphi_i)$ with $U_i = \{(x_0:\ldots:x_m)\mid x_i\neq 0\}$ and $\varphi_i(x_0:\ldots:x_m) = (\tfrac{x_0}{x_i},\ldots,\tfrac{x_{i-1}}{x_i},\tfrac{x_{i+1}}{x_i}, \ldots, \tfrac{x_m}{x_i})$. A Riemannian structure on $\mathbb{P}(\R^n)$ is the \emph{Fubini--Study metric}; see, e.g., \cite[Section 14.2.2]{condition}: the tangent space to $x$ can be identified with
\begin{equation}\label{tangent_proj}
  \Tang{ x}{\mathbb{P}(\R^n)\cong \vect{x}^\perp}, \quad\text{ where $\vect x \in x$ is a representative;}
\end{equation}
and through this identification the Fubini--Study metric is $g(\vect y_1,\vect y_2):=\frac{\langle \vect y_1,\vect y_2\rangle}{\Vert \vect x\Vert}$. The \emph{Fubini--Study} distance~$d_\Pj$ is the distance associated to the Fubini--Study metric. For points $x,y \in \mathbb{P}(\R^n)$ the formula is
  $$d_\Pj(x,y) = \frac{\vert \langle \vect x,\vect y\rangle\vert}{\Vert \vect x\Vert \Vert \vect y\Vert}, \quad \text{ where $\vect x\in x$, $\vect y\in y $ are representatives.}$$
For the Fubini--Study distance in $\Pj(\R^{n_1})\times\cdots\times \Pj(\R^{n_d})$ we write
\begin{equation}\label{product_fubini_study}
\mathrm{dist}_\Pj((x_1,\ldots,x_d),(y_1,\ldots,y_d)) := \sqrt{\sum_{i=1}^d d_\Pj(x_i,y_i)^2}.
\end{equation}

\begin{figure}[tb]\small
\begin{center}
\begin{tikzpicture}[scale=.8,thick,x=1cm]
  \draw (3,0) circle (2);
  \draw (-3,0) circle (1);
  \draw[red, very thick] (3,2) arc (90:135:2);
  \draw[black,dashed] (3,0) -- (1,2);
  \draw[black,->] (3,0) -- (3,2);
  \draw[black,->] (3,2) -- (1,2);
  \draw[red, very thick] (-3,1) arc (90:135:1);
  \draw[black,dashed] (-3,0) --(-4,1);
  \draw[black,->] (-3,0) -- (-3,1);
  \draw[black,->] (-3,1) -- (-4,1);
  \draw[] (-2.7,0.5) node {$x^1$};
  \draw[] (-3.275,1.30) node {$\Delta x^1$};
  \draw[] (-3.15,0.4) node {{\footnotesize $\phi$}};
  \draw[] (3.30,1) node {$x^2$};
  \draw[] (2,2.30) node {$\Delta x^2$};
  \draw[] (2.8,0.5) node {{\footnotesize $\phi$}};
  \draw[] (0,-2.5) node {$\tan \phi = \frac{\Norm{\Delta {x}^2}}{\Norm{{x}^1}}=\frac{\Norm{\Delta {x}^2}}{\Norm{{x}^2}}$};
\end{tikzpicture}
\end{center}
\caption{The picture depicts relative errors in the weighted distance, where ${x}^1 \in \Pj(\R^{n_1})$ and ${x}^2 \in \Pj(\R^{n_2})$ with $n_1 > n_2$. The relative errors of the tangent directions $\Delta {x}^1$ and $\Delta {x}^2$ are both equal to $\tan \phi$, but the contribution to the weighted distance marked in red is larger for the large circle, which corresponds to the smaller projective space $\Pj(\R^{n_2})$.}
\label{fig:weighted_distance}
\end{figure}
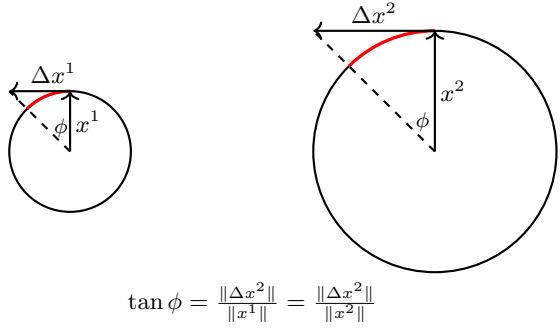

The weighted distance, which is the protagonist of \refthm{cnt}, is introduced next.
\begin{dfn}[Weighted distance]\label{weighted_distance}
The \emph{weighted distance} between two points $p=(p_1,\ldots,p_d)$ and $q=(q_1,\ldots,q_d)\in \mathbb{P}(\R^{n_1})\times \cdots \times \mathbb{P}(\R^{n_d})$ is defined as
$$
d_\mathrm{w}(p,q) :=  \sqrt{\sum_{i=1}^d (n-n_i) d_\mathbb{P}(p_i,q_i)^2},
$$
where, as before, $n=\dim \Var S$.
The weighted distance on $\Var S^{\times r}$ then is defined as
$$\dist_\mathrm{w}((\tensor A_1,\ldots, \tensor A_r),(\tensor B_1,\ldots, \tensor B_r)) : = \sqrt{\sum_{i=1}^r d_\mathrm{w}(\sigma^{-1}(\tensor A_i), \sigma^{-1}(\tensor B_i))^2},$$
where $\sigma^{-1}$ is the inverse of the projective Segre map from \cref{eqn_def_sigma}.
\end{dfn}
For $n_1 > n_2$ the relative errors in the factor~$\mathbb{P}(\R^{n_2})$ weigh more than relative errors in the factor $\mathbb{P}(\R^{n_1})$ when the measure is the weighted distance $d_\mathrm{w}$; this is illustrated in \cref{fig:weighted_distance}.

\section{The expected value of the condition number}\label{sec_expected_value}

Before proving \cref{thm_expected_value}, we need four auxiliary lemmata. The first provides a deterministic lower bound of the condition number.
\begin{lemma}\label{lemma_bound}
Let $r\geq 1$. For rank-1 tuples $(\tensor{A}_1,\ldots,\tensor{A}_r)$ in $\mathbb{R}^{n_1\times \cdots\times n_d}$ we have $\kappa(\tensor{A}_1,\ldots,\tensor{A}_r)\geq 1$.
\end{lemma}
\begin{proof}
The condition number equals the inverse of the smallest singular value of a matrix all of whose columns are of unit length by \cite[Theorem 1.1]{BV2017}. The result follows from the min-max characterization of the smallest singular value.
\end{proof}

The next lemma is a basic computation in Riemannian geometry.
\begin{lemma}\label{prop:inversedistance}
	Let $M$ be a Riemannian manifold, and $N$ a codimension $c$ submanifold of $M$. Let $\mathrm{dist}_M$ denote the Riemannian distance on $M$ and $\vert\omega\vert$ be the density on $M$. Then,
	\[
	\int_{x\in M}\left(\frac{1}{\mathrm{dist}_M(x,N)}\right)^c\,\vert\omega\vert=\infty.
	\]
\end{lemma}
\begin{proof}
	Let $m,k$ be the dimensions of $M,N$ and let $y\in N$ be any point. Let $\epsilon >0$. From the definition of being a submanifold, there exists an open neighborhood $U$ of $y$ in $M$ and a diffeomorphism $\phi:U\to B_\epsilon(\R^k)\times B_\epsilon(\R^{m-k})$, such that $N\cap U=\phi^{-1}(B_\epsilon(\R^k)\times \{0\})$, where $B_\epsilon(\R^m)$ is the open ball of radius $\epsilon$ in $\R^m$. By compactness, choosing $\epsilon$ small enough, we can assume that there is a positive constant $C$ such that the derivative of $\phi$ satisfies
	$\|\deriv{\phi}{x}\|\leq C$, $\|\deriv{\phi}{x}^{-1}\|\leq C$, and  $|\det(\deriv{\phi}{x})|\ge C$ for all $x\in U$. In particular, the length $L$ of a curve in $U$ and the length $L'$ of its image under $\phi$ satisfy $L\leq CL'$. Writing $(\vect{x}_1,\vect{x}_2):=\phi(x)$ for the image of $x$ under $\phi$ we thus have $\mathrm{dist}_M(x,N)\leq  C \|\vect{x}_2\|.$
	The change of variables theorem, i.e., \cite[Theorem 3-13]{Spivak1965}, gives
	\begin{align*}
	\int_{x\in U}\left(\frac{1}{\mathrm{dist}_M(x,N)}\right)^c\,\vert\omega\vert&= \int_{x\in U}\left(\frac{1}{\mathrm{dist}_M(x,N)}\right)^{m-k}\,\vert\omega\vert\\
  &\geq \frac{1}{C^{m-k+1}}\int_{(\vect x_1,\vect x_2)\in B_\epsilon(\R^k)\times B_\epsilon(\R^{m-k})}\frac{1}{\|\vect x_2\|^{m-n}}\,\mathrm{d}\vect x_1\mathrm{d}\vect x_2.
\end{align*}
	Up to positive constants, using Fubini's theorem, i.e., \cite[Theorem 3-10]{Spivak1965}, and passing to polar coordinates, this last integral equals
	\[
	\int_{\vect x_1\in B_\epsilon(\R^k)}\int_{0}^\epsilon\frac{t^{m-k-1}}{t^{m-k}}\,\mathrm{d}t\,dv=\infty.
	\]
	The lower bound for the integral in the lemma then follows from
	\[
\int_{x\in M}\left(\frac{1}{\mathrm{dist}_M(x,N)}\right)^c\,\vert\omega\vert\geq	\int_{x\in U}\left(\frac{1}{\mathrm{dist}_M(x,N)}\right)^c\,\vert\omega\vert
	\]
This finishes the proof.
\end{proof}

Inspecting \cref{cnt}, we see that combining it with the above lemma contains the key idea for proving that the expected value of the condition number can be infinite. However, to use these results in our proof of \cref{thm_expected_value}, we need to ensure that \cref{prop:inversedistance} applies. \Cref{cnt} uses the weighted distance from \cref{weighted_distance} and it is not immediately evident whether it is induced by a Riemannian metric on $\mathbb{P}(\R^{n_1})\times \cdots \times \mathbb{P}(\R^{n_d})$.
Fortunately, the next lemma shows that it is.

\begin{lemma}
Let $\langle \cdot , \cdot \rangle$ be the Fubini--Study metric. We define the weighted inner product~$\langle \cdot, \cdot \rangle_{\mathrm{w}}$ on the tangent space at $p\in \mathbb{P}(\R^{n_1})\times \cdots \times \mathbb{P}(\R^{n_d})$ as follows. For $\vect{u}, \vect{v}\in \Tang{p}{(\mathbb{P}(\R^{n_1})\times \cdots \times \mathbb{P}(\R^{n_d}))}$, $\vect{u}=(\vect u^1,\ldots,\vect u^d)$, $\vect{v}=(\vect v^1,\ldots,\vect v^d)$, we define
  $
  \langle \vect u,\vect v\rangle_{\mathrm{w}}
  := \sum_{i=1}^d (n - n_i) \langle \vect u^i,\vect v^i\rangle
  $.
Then, the distance on $\mathbb{P}(\R^{n_1})\times \cdots \times \mathbb{P}(\R^{n_d})$ corresponding to~$\langle \cdot, \cdot \rangle_{\mathrm{w}}$ is $d_\mathrm{w}$.
\end{lemma}
\begin{proof}
Let $\gamma(t) = (\gamma_1(t),\ldots,\gamma_d(t))$ be a piecewise continuous curve in $\mathbb{P}(\R^{n_1})\times \cdots \times \mathbb{P}(\R^{n_d})$ connecting $p,q\in \mathbb{P}(\R^{n_1})\times \cdots \times \mathbb{P}(\R^{n_d})$, such that the distance between $p,q$ given by $\langle \cdot ,\cdot \rangle_\mathrm{w}$ is
$$
\int_0^1 \langle \gamma'(t), \gamma'(t) \rangle_\mathrm{w}^\frac{1}{2}\, \d t
= \int_0^1 \left(\sum_{i=1}^d(n-n_i)\langle \gamma_i'(t), \gamma_i'(t) \rangle\right)^\frac{1}{2} \,\d t.
$$
Because $(n-n_i)\langle \gamma_i'(t) ,\gamma_i'(t) \rangle = \langle \sqrt{n-n_i}\,\gamma_i'(t), \sqrt{n-n_i}\,\gamma_i'(t) \rangle$ and because we have the identity of tangent spaces $\Tang{\gamma_i(t)}{\mathbb{P}(\R^{n_i})} = \Tang{\gamma_i(t)}{\mathbb{S}(\R^{n_i})}$ for all $i$ and $t$, we may view $\gamma$ as the shortest path between two points on a product of $d$ spheres with radii $\sqrt{n-n_1},\ldots, \sqrt{n-n_d}$. The length of this path is $d_\mathrm{w}(p,q)$.
\end{proof}

Let $\sigma$ be the projective Segre map from \cref{eqn_def_sigma}. By \cite[Section 4.3.4.]{Landsberg2012}, $\sigma$ is a diffeomorphism and we define a Riemannian metric $g$ on $\mathbb{P}\Var{S}$ to be the pull-back metric of $\langle \cdot, \cdot \rangle_{\mathrm{w}}$ under~$\sigma^{-1}$; see \cite[Proposition 13.9]{Lee2013}. Then, by construction, we have the following result.

\begin{cor}
  The weighted distance $\dist_\mathrm{w}$ on $\Pj\Var S^{\times r}$ is given by the Riemannian metric $g$.
\end{cor}

The last technical lemma we need is the following.

\begin{lemma}\label{sigma_det}
Consider the projective Segre map $\sigma : \mathbb{P}(\R^{n_1})\times \cdots \times \mathbb{P}(\R^{n_d}) \to \mathbb{P}\Var{S}$ from \cref{eqn_def_sigma}. For any point $p=([\sten{a}{1}{}],\ldots, [\sten{a}{d}{}]) \in \mathbb{P}(\R^{n_1})\times \cdots \times \mathbb{P}(\R^{n_d})$ we have $\vert\det(\deriv{\sigma}{p})\vert=1$.
\end{lemma}
\begin{proof}
We denote by $\vect{e}_i^j$ the $i$th standard basis vector of $\R^{n_j}$; i.e., $\vect{e}_i^j$ has zeros everywhere except for the $i$th entry, where it has a 1. To ease notation, let us assume $\vect{e}_i^j$ to be a row vector.
Because each $\Pj(\R^{n_j})$ is an orbit of $[\vect{e}_1^j]$ under the orthogonal group, it suffices to show the claim for $p=([\vect{e}_1^1],\ldots, [\vect{e}_1^d])$. By \cref{tangent_proj}, an orthonormal basis for the tangent space $\Tang{[\vect{e}_1^j]}{\Pj(\R^{n_j})}$ is $\{\vect{e}_2^j,\ldots,\vect{e}_{n_j}^j\}$. Hence, an orthonormal basis for $\Tang{p}{(\Pj(\R^{n_1})\times\cdots\times\Pj(\R^{n_d}))}$ is
$$
\bigcup_{j=1}^d\{(\underbrace{0,\ldots,0}_{j-1\text{ times}},\vect{e}_i^j,\underbrace{0,\ldots,0}_{d-j+1\text{ times}}) \mid 2\leq i\leq n_j\}.
$$
Fix $1\leq j\leq d$ and $2\leq i\leq n_j$. Then, by the product rule, we have
  $$
  \deriv{\sigma}{p}(0,\ldots,0,\vect e_j,0,\ldots,0) = \vect{e}_1^1 \otimes \cdots \otimes \vect{e}_1^{j-1} \otimes \vect{e}_i^j \otimes \vect{e}_1^{j+1} \otimes \cdots \otimes \vect{e}_1^d.
  $$
It is easily verified that $\{ \vect{e}_1^1 \otimes \cdots \otimes \vect{e}_1^{j-1} \otimes \vect{e}_i^j \otimes \vect{e}_1^{j+1} \otimes \cdots \otimes \vect{e}_1^d \mid 1\leq j\leq d, 2\leq i\leq n_j\}$ is an orthonormal basis of $\Tang{\sigma(p)}{\mathbb{P}\Var S}$ (for instance, by using \cref{inner_droducts} below). This shows that $\deriv{\sigma}{p}$ maps an orthonormal basis to an orthonormal basis. Hence, $\vert\det(\deriv{\sigma}{p})\vert=1$.
\end{proof}

\begin{rem}
In fact, the proof of the foregoing lemma shows more than $\vert\det(\deriv{\sigma}{p})\vert=1$. Namely, it shows that $\sigma$ is an \textit{isomety} in the sense of \cref{appendix:isometry_dfn}.
\end{rem}

Now we have gathered all the ingredients to prove \refthm{thm_expected_value}.

\begin{proof}[Proof of \refthm{thm_expected_value}]
First, we use that the condition number is \emph{scale invariant}. That is, for all $t_1,\ldots,t_r\in \R\backslash\{0\}$ we have by \cite[Proposition 4.4]{BV2017}:
$$
\kappa(t_1\tensor{A}_1,\ldots, t_r\tensor{A}) = \kappa(\tensor{A}_1,\ldots, \tensor{A}).
$$
This implies that the random variable under consideration is independent of the scaling of the factors $\sten{a}{i}{j}$ and, consequently, we have (see, e.g., \cite[Remark 2.24]{condition})
$$
\mean\limits_{\sten{a}{i}{j} \in \R^{m_j}, 1\leq j\leq d, 1\leq i\leq r\atop \text{ standard normal i.i.d.}}\, \bigl[ \kappa(\tensor{A}_1,\ldots, \tensor{A}_r)^c \bigr] = \mean\limits_{\sten{a}{i}{j} \in \Pj(\R^{m_j}), 1\leq j\leq d, 1\leq i\leq r \atop \text{ uniformly i.i.d.}}\, \bigl[ \kappa(\tensor{A}_1,\ldots, \tensor{A}_r)^c \bigr].
$$
Let $\vert\omega\vert$ denote the density on $\Var S^{\times r}=\Var S_{n_1,\ldots,n_d}^{\times r}$. By \cref{sigma_det}, the Jacobian of the change of variables via the projective Segre map $\sigma$ is constant and equal to 1. Hence,
$$
\mean\limits_{\sten{a}{i}{j} \in \Pj(\R^{m_j}), 1\leq j\leq d, 1\leq i\leq r \atop \text{ uniformly i.i.d.}}\, \bigl[ \kappa(\tensor{A}_1,\ldots, \tensor{A}_r)^c \bigr]
= \frac{1}{C} \int_{(\tensor{A}_1,\ldots, \tensor{A}_r)\in\Pj \Var S^{\times r}} \kappa(\tensor{A}_1,\ldots, \tensor{A}_r)^c\, \vert\omega\vert,
$$
where $C=\int_{\Pj \Var S^{\times r}}  \vert\omega\vert<\infty$, because $\Pj \Var S^{\times r}$ is compact. For brevity, we write $\tuple{p} = (\tensor{A}_1, \ldots, \tensor{A}_r)$.
Then, by \cref{cnt} we have
$$
\int_{\tuple{p}\in\Pj \Var S^{\times r}} \kappa(\tuple{p})^c \, \vert\omega\vert
\geq \int_{\tuple{p}\in\Pj \Var S^{\times r}} \left( \frac{1}{\mathrm{dist}_\mathrm{w}(\tuple{p}, \Sigma_\Pj)} \right)^c \, \vert\omega\vert.
$$
We cannot directly apply \cref{lem_codim_one} here, because the weighted distance $\mathrm{dist}_\mathrm{w}$ \emph{is not given by the product Fubini--Study metric}. However, from the definitions of the weighted distance and the Fubini--Study distance \cref{product_fubini_study}, we find $\mathrm{dist}_\mathrm{w}(\tuple{p}, \Sigma_\Pj)\leq \sqrt{n}\,\mathrm{dist}_\Pj(\tuple{p}, \Sigma_\Pj))$. Therefore, we have
$$
\int_{\tuple{p}\in\Pj \Var S^{\times r}} \left( \frac{1}{\mathrm{dist}_\mathrm{w}(\tuple{p}, \Sigma_\Pj)} \right)^c\, \vert\omega\vert
\ge
n^{-\frac{c}{2}} \int_{\tuple{p}\in\Pj \Var S^{\times r}} \left( \frac{1}{\mathrm{dist}_\Pj(\tuple{p}, \Sigma_\Pj)} \right)^c \, \vert\omega\vert.
$$
By assumption, there is a manifold $U\subset \Sigma_\Pj$ of codimension $c$ in $\Var S^{\times r}$
Applying \cref{prop:inversedistance} to this manifold we have
\[
     	\int_{\tensor{A}_1,\ldots,\tensor{A}_r\in \Pj\Var{S}} \left( \frac{1}{\mathrm{dist}_\Pj((\tensor{A}_1,\ldots,\tensor{A}_r),\Sigma_\Pj)} \right)^c\,\vert\omega\vert
     	\geq \int_{\tensor{A}_1,\ldots,\tensor{A}_r\in \Pj\Var{S}} \left( \frac{1}{\mathrm{dist}_\Pj((\tensor{A}_1,\ldots,\tensor{A}_r),U)} \right)^c\,\vert\omega\vert=\infty.
\]
Putting all the equalities and inequalities together, we therefore get
$$
\mean\limits_{\sten{a}{i}{j} \in \R^{m_j}, 1\leq j\leq d, 1\leq i\leq r\atop \text{ standard normal i.i.d.}}\, \bigl[ \kappa(\tensor{A}_1,\ldots, \tensor{A}_r)^c \bigr]  = \infty.$$
By \cref{lemma_bound}, the condition number satisfies $\kappa( \tensor{A}_1,\ldots, \tensor{A}_r ) \ge 1$ for every $(\tensor{A}_1,\ldots,\tensor{A}_r)\in\Var{S}^{\times r}$. This together with the foregoing equation implies for $c \le e$:
$$\mean\limits_{\sten{a}{i}{j} \in \R^{m_j}, 1\leq j\leq d, 1\leq i\leq r\atop \text{ standard normal i.i.d.}}\, \bigl[ \kappa(\tensor{A}_1,\ldots, \tensor{A}_r)^e \bigr] = \infty.
$$
The proof is finished.
\end{proof}


Next, we investigate a particular corollary of the foregoing result. We will show that for third-order tensors $\R^{n_1 \times n_2 \times n_3}$, $n_1 \ge n_2 \ge n_3 \ge 2$, the expected value of  $(n_3-1)$th power of the condition number of random rank-$r$ tensors is indeed $\infty$. The following is the key ingredient.

\begin{lemma} \label{lem_codim_one}
Let $\Var{S}$ be the Segre manifold in $\R^{n_1\times n_2 \times n_3}$, $n_1\ge n_2\ge n_3 \ge 2$, and let $\Sigma_\Pj \subset (\Pj\Var{S})^{\times r}$ be the ill-posed locus.
Then, there is a subvariety~$\Var V\subset \Sigma_\Pj$ of codimension $n_3-1$ in $(\Pj\Var S)^{\times r}$.
\end{lemma}
\begin{proof}
Consider the regular map
\begin{align*}
 \psi : \left(\Pj(\R^{n_1}) \times \Pj(\R^{n_2}) \times \Pj(\R^{n_3}))^{\times r-1} \times \Pj(\R^{n_1}\right) \times \Pj(\R^{n_2}) &\to (\Pj\Var{S})^{\times r} \\
 \bigl( ([\sten{a}{i}{}], [\sten{b}{i}{}], [\sten{c}{i}{}])_{i=1}^{r-1}, ([\sten{a}{r}{}], [\sten{b}{r}{}]) \bigr) &\mapsto \bigl( ([\sten{a}{i}{} \otimes \sten{b}{i}{} \otimes \sten{c}{i}{}])_{i=1}^{r-1}, [\sten{a}{r}{} \otimes \sten{b}{r}{} \otimes \sten{c}{1}{}] \bigr).
\end{align*}
The image of $\psi$, write $\Var{V} = \operatorname{Im}(\psi)$, is a projective variety by \cite[Theorem 3.13]{Harris1992}. Because the projective Segre map from \cref{eqn_def_sigma} is a bijection, the fiber of $\psi$ at any point in $\Var{V}$ consists of precisely one point. As a result, by \cite[Theorem 11.12]{Harris1992}, $\dim \Var{V}$ equals the dimension of the source, which is seen to be $r(\dim \Pj\Var{S}) - n_3+1$, i.e., $\operatorname{codim}(\Var{V})=n_3-1$.

Next, we show that $\Var{V} \subset \Sigma_\Pj$, which then concludes the proof. Let $[\tensor{A}_i] = [\sten{a}{i}{}\otimes\sten{b}{i}{}\otimes\sten{c}{i}{}]$ be such that $([\tensor{A}_1], \ldots, [\tensor{A}_r]) \in \Var{V}$. Thus, $[\sten{c}{r}{}] = [\sten{c}{1}{}]$.
Consider the (affine) tangent spaces
\begin{align*}
\Tang{\tensor{A}_1}{\Var{S}}
= \Tang{\sten{a}{1}{}\otimes\sten{b}{1}{}\otimes\sten{c}{1}{}}{\Var{S}}
&= \R^{n_1} \otimes \sten{b}{1}{} \otimes \sten{c}{1}{} + \sten{a}{1}{}\otimes (\sten{b}{1}{})^\perp \otimes \sten{c}{1}{} + \sten{a}{1}{} \otimes \sten{b}{1}{} \otimes (\sten{c}{1}{})^\perp, \text{ and } \\
\Tang{\tensor{A}_r}{\Var{S}}
= \Tang{\sten{a}{r}{}\otimes\sten{b}{r}{}\otimes\sten{c}{1}{}}{\Var{S}}
&= (\sten{a}{r}{})^\perp \otimes \sten{b}{r}{} \otimes \sten{c}{1}{} + \sten{a}{r}{} \otimes \R^{n_2} \otimes \sten{c}{1}{} + \sten{a}{r}{} \otimes \sten{b}{r}{} \otimes (\sten{c}{1}{})^\perp.
\end{align*}
They intersect at least in the $1$-dimensional subspace
\(
 \{\alpha \,\sten{a}{r}{} \otimes \sten{b}{1}{} \otimes \sten{c}{1}{}\mid \alpha \in \R\}
\). This means that
\[
\dist_{\mathrm{P}}((\Tang{\tensor{A}_1}{\Var{S}},\ldots, \Tang{\tensor{A}_r}{\Var{S}}) ,\Sigma_\Gr) = 0;
\]
hence, by \cref{characterization123}, $\kappa(\tensor{A}_1,\ldots,\tensor{A}_r) = \infty$ and so $([\tensor{A}_1],\ldots,[\tensor{A}_r]) \in \Sigma_\Pj$.
\end{proof}

We can now wrap up the proof of \cref{cor_expected_value}.

\begin{proof}[Proof of \cref{cor_expected_value}]
\Cref{lem_codim_one} shows there is a subvariety $\Var{V} \subset \Sigma_\Pj$ with codimension equal to $n_3-1$. Let $p$ be any smooth point in this subvariety, and consider a neighborhood~$U$ of $p$ in $(\Pj\Var{S})^{\times r}$ such that all points in $U$ are smooth points of $\Var{V}$. Then, $U$ is a submanifold of~$\Sigma_\Pj$ that has codimension $n_3-1$ in $\Var S^{\times r}$. Hence, \cref{thm_expected_value} applies and \cref{cor_expected_value} is proven.
\end{proof}


\Cref{lem_codim_one} still leaves some doubt over the precise codimension of $\Sigma_\Pj$ in other tensor formats than $n_1\times n_2\times 2$. It might be possible to sharpen \cref{cor_expected_value}. Namely, if there exists a submanifold $\Var{M}$ of codimension $k < n_3-1$ in $(\Pj\Var{S})^{\times r}$ with $\Var{M} \subset \Sigma_\Pj$, then we also have $\mathbb{E}[ \kappa(\tensor{A}_1,\ldots,\tensor{A}_r)^k ] = \infty$. For small tensors, we can compute the codimension of the ill-posed locus using computer algebra software. Employing Macaulay2 \cite{M2}, we were able to show that \cref{lem_codim_one} cannot be improved for small tensors with rank $r=2$.

\begin{prop} \label{prop_impossible}
Let $\Var{S}$ be the Segre manifold in $\R^{n_1 \times n_2 \times n_3}$, $10 \ge n_1 \ge n_2 \ge n_3 \ge 2$, and let $\Sigma_\Pj \subset (\Pj\Var{S})^{\times 2}$ be the ill-posed locus. There is no subvariety $\Var V\subset \Sigma_\Pj$ of codimension $k < n_3-1$.
\end{prop}
\begin{proof}
It is an exercise to verify that the Segre manifold $\Var{S}$ is covered by the charts $(U_{i,j}, \phi_{i,j})$, defined uniquely as follows: $U_{i,j} := \operatorname{Im}( \phi_{i,j}^{-1} )$ and
\begin{align*}
 \phi_{i,j} : \,&\R^{n_1-1} \times \R^{n_2-1} \times \R^{n_3} \to \R^{n_1 \times n_2 \times n_3}, \\
 &(\vect{x}, \vect{y}, \vect{z}) \mapsto
 (x_1, \ldots, x_{i-1}, 1, x_{i+1}, \ldots, x_{n_1-1}) \otimes (y_1, \ldots, y_{j-1}, 1, y_{j+1}, \ldots, y_{n_2-1}) \otimes \vect{z}.
\end{align*}
Let $p_1 \in U_{i_1,j_1}$ and $p_2 \in U_{i_2,j_2}$ and $\tensor{A}_1 = \phi_{i_1,j_1}(p_1)$, $\tensor{A}_2 = \phi_{i_2,j_2}(p_2)$. The corresponding rank-2 tensor is $\Phi(\tensor{A}_1, \tensor{A}_2) = \tensor{A}_1+ \tensor{A}_2$. By definition of the derivative of the addition map $\Phi$, its matrix with respect to an orthonormal basis for $\phi_{i_1,j_1}( U_{i_1,j_1} ) \times \phi_{i_2,j_2}( U_{i_2,j_2} )$ and the standard basis on $\R^{n_1 \times n_2 \times n_3} \simeq \R^{n_1 n_2 n_3}$ is the Jacobian of the transformation $\Phi \circ (\phi_{i_1,j_1} \times \phi_{i_2,j_2})$; see \cite[pages 55--65]{Lee2013}. For example, if $i_1=j_1=i_2=j_2$ and  $n_1=n_2=n_3=2$, then the derivative $\deriv{\Phi}{(\tensor{A}_1,\tensor{A}_2)}$ is represented in bases as the $8 \times 8$ Jacobian matrix of the map from
\(
 (\R^1 \times \R^1 \times \R^2) \times (\R^1 \times \R^1 \times \R^2) \to \R^8
\) taking
\[
 ( a_2, b_2, c_1, c_2) \times (x_2, y_2, z_1, z_2 ) \mapsto \begin{bmatrix} 1 \\ a_2 \end{bmatrix} \otimes \begin{bmatrix} 1 \\ b_2 \end{bmatrix} \otimes \begin{bmatrix} c_1 \\ c_2 \end{bmatrix} + \begin{bmatrix} 1 \\ x_2 \end{bmatrix} \otimes \begin{bmatrix} 1 \\ y_2 \end{bmatrix} \otimes \begin{bmatrix} z_1 \\ z_2 \end{bmatrix}.
\]
The ill-posed locus is then the projectivization of the locus where these Jacobian matrices have linearly dependent columns. Note that the codimension of $\Sigma_\Pj \in (\Pj\Var{S})^{\times 2}$ is the same as the codimension in $\Var{S}^{\times 2}$ of the affine cone over $\Sigma_\Pj$. The codimension of the variety where these Jacobian matrices are not injective is the number we need to compute. This variety is given by the vanishing of all maximal minors.

Let $s = n_1+n_2+n_3-2 = \dim \Var{S}$. Computing all $\binom{n_1 n_2 n_3}{2s}$ maximal minors of a Jacobian matrix~$J$ is too expensive. Instead we proceed as follows.
Note that we can perform all computations over $\mathbb{Q}$, because the Jacobian matrix is given by polynomials with integer coefficients. By homogeneity, we can always assume that the first rank-$1$ tensor is $p_1 = \sten{e}{1}{1} \otimes \sten{e}{1}{2} \otimes \sten{e}{1}{3} \in \Var{S}$, where $\sten{e}{1}{j} \in \mathbb{Q}^{n_j}$ is the first standard basis vector. For each chart on the second copy of $\Var{S}$, we then take $p_2 \in U_{i_2,j_2}$ and construct the Jacobian matrix~$J$. We then multiply it with the column vector $\vect{k} = (k_1, k_2, \ldots, k_{s}) \in \mathbb{Q}^{s}\setminus\{0\}$ consisting of free variables; note that $\mathbb{Q}^{s}\setminus\{0\}$ should be covered by charts $V_i$ for this. Now, the condition number $\kappa(p_1,p_2) = \infty$ if $\vect{v} := J \vect{k}$ is zero, as then there would be a nontrivial kernel. It follows that the ideal generated by the maximal minors of $J$ is then equal to the elimination ideal obtained by eliminating the $k_i$'s from the ideal generated by the $n_1 n_2 n_3$ components of $\vect{v}$. This can be computed more efficiently in Macaulay2 than generating all maximal minors. The ideal thusly obtained is the same ideal as the one that would have been begotten by performing all computations over $\R$, by the elementary properties of computing Gr\"obner bases \cite[Chapters 2--3]{CLO2015}. Performing this computation in all charts and taking the minimum of the computed codimensions, we found in all cases the value $n_3-1$.
\end{proof}

\section{The condition number and distance to  ill-posedness}\label{sec:proof}

In the course of establishing that the expected value of powers of the condition number can be infinite, that is \cref{thm_expected_value}, we relied on the unproved \cref{cnt}.
The overall goal of this section is to prove \refthm{cnt}. We start with a short detour and recall some results from Riemannian geometry.

\subsection{Isometric immersions}\label{sec:isometric}
 Recall that a smooth map $f : M \to N$ between manifolds $M, N$ is called a \textit{smooth immersion} if the derivative $\deriv{f}{p}$ is injective for all $p \in M$; see \cite[Chapter 4]{Lee2013}. Hence, $\dim M \leq \dim N$.

 \begin{dfn}\label{appendix:isometry_dfn}
 A differentiable map $f: {M}\to {N}$ between Riemannian manifolds $(M,g)$, $(N,h)$ is called an \emph{isometric immersion} if $f$ is a smooth immersion and, furthermore, for all $p\in M$ and $u,v\in\Tang{p}{M}$ it holds that $g_p(u,v) = h_{f(p)}(\deriv{f}{p} (u), \deriv{f}{p} (v))$. If in addition $f$ is a diffeomorphism then it is called an \emph{isometry}.
 \end{dfn}

 We will need the following lemma.

 \begin{lemma}\label{appendix:isometry_importan_lemma}
 Let $M,N, P$ be Riemannian manifolds and $f:M\to N$ and $g:N\to P$ be differentiable maps.
 \begin{enumerate}
   \item Assume $f$ is an isometry. Then, $g\circ f$ is an isometric immersion if and only if $g$ is an isometric immersion.
   \item Assume $g$ is an isometry. Then, $g\circ f$ is an isometric immersion if and only if $f$ is an isometric immersion.
   \item If $f$ is an isometric immersion, then for all $p,q\in M$: $\dist_M(p,q)\geq \dist_N(f(p),f(q)).$
   \end{enumerate}
 \end{lemma}
 \begin{proof}
 Let $p\in M$. By the chain rule we have $\deriv{(g\circ f)}{p}=\deriv{g}{f(p)}\;\deriv{f}{p}$. Hence, for all $\vect{u}, \vect{v} \in \Tang{p}{M}$  we have
 $\langle \deriv{(g\circ f)}{p}\;\vect{u}, \deriv{(g\circ f)}{p} \;\vect{v} \rangle  =  \langle \deriv{g}{f(p)}\;\deriv{f}{p}\; \vect{u}, \deriv{g}{f(p)}\;\deriv{f}{p}\; \vect{v} \rangle.$
 We prove (1): If $g$ is isometric, the foregoing equation simplifies to $\langle \deriv{(g\circ f)}{p}\;\vect{u}, \deriv{(g\circ f)}{p} \;\vect{v} \rangle =  \langle \deriv{f}{p}\; \vect{u}, \deriv{f}{p}\; \vect{v} \rangle=  \langle  \vect{u}, \vect{v} \rangle$. Hence, $g\circ f$ is isometric. By the same argument, if $g\circ f$ is isometric, $g=g\circ f \circ f^{-1}$ is isometric.
 The second assertion is proved similarly.
 Finally, the last assertion is immediately clear from the definition of Riemannian distance.
 \end{proof}

\subsection{Proof of \refthm{cnt}}
In the introduction we recalled, in \refeqn{characterization123}, that the condition number is equal to the inverse distance of the tuple of tangent spaces to the tuples of linear spaces not in general position. The idea to prove \refthm{cnt} is to make use of  \reflem{appendix:isometry_importan_lemma}~(3) from the previous subsection. This lemma lets us to compare Riemannian distances between two manifolds. However,
the projection distance from \cref{projection_distance} is not given by some Riemannian metric on $\Gr(\Pi,n)$. In fact, up to scaling there is a unique orthogonally invariant metric on $\Gr(\Pi,n)$ when $\Pi~>~4$; see \cite{leichtweiss}. A usual choice of scaling is such that the distance associated to the metric is given~by $d(V,W)=\sqrt{\theta_1^2+\cdots+\theta_n^2}$, where $\theta_1,\ldots,\theta_n$ are the \emph{principal angles} between $V$ and $W$ \cite{principal_angles}. Let us call this choice of metric the \emph{standard metric} on $\Gr(\Pi,n)$. From this we construct the following distance function on $\Gr(\Pi,n)^{\times r}$:
  \begin{equation}\label{dist_R}
  \dist_\mathrm{R}((V_i)_{i=1}^r,(W_i)_1^r):= \sqrt{\sum_{i=1}^r d(V_i,W_i)^2}.
  \end{equation}
We can also express the projection distance in terms of the principal angles between the linear spaces $V$ and~$W$: $\Vert \pi_{V}-\pi_{W}\Vert = \max_{1\leq i\leq n} \vert\sin \theta_i\vert$; see, e.g., \cite[Table 2]{YL2016}. Since, for all $-\frac{\pi}{2}<\theta<\frac{\pi}{2}$ we have $\vert\sin(\theta)\vert\leq\vert\theta\vert$, this shows that
\begin{equation}\label{part1:cnt_distances}
   \dist_\mathrm{P}((V_i)_{i=1}^r,(W_i)_{i=1}^r)\leq \dist_\mathrm{R}((V_i)_{i=1}^r,(W_i)_{i=1}^r)
\end{equation}
This is an important inequality because it allows us to prove \cref{cnt} by replacing $\dist_\mathrm{P}$ by~$\dist_\mathrm{R}$. The second key result for the proof of \cref{cnt} is the following.
\begin{prop} \label{isometry_thn}
We consider to $\mathbb{P}\Var{S}$ to be endowed with the weighted metric from \cref{weighted_distance} and $\Gr(\Pi,n)$ to be endowed with the standard metric. Then,
  $\phi: \mathbb{P}\Var{S} \to \Gr(\Pi,n), [\tensor A] \mapsto \Tang{\tensor A}{\Var{S}}$
is an isometric immersion in the sense of Definition \ref{appendix:isometry_dfn}.
\end{prop}
\begin{rem}
In the proposition $\phi$ is \emph{not} the Gauss map
$\Pj \Var{S} \to \Gr(n-1,\Pj\R^\Pi), [\tensor{A}] \mapsto [\Tang{\tensor{A}}{\Var{S}}],$
which maps a tensor to a \emph{projective} subspace of $\Pj\R^\Pi$ of dimension $n-1 = \dim \Pj\Var{S}$.
\end{rem}

\Cref{isometry_thn} lies at the heart of this section, but its proof is quite technical and is therefore delayed until \cref{sec_proof_isometry} below. First, we use it to give a proof of \cref{cnt}.

\begin{proof}[Proof of Theorem \ref{cnt}]
Assume that $\Gr(\Pi,n)^{\times r}$ is endowed with the standard metric on $\Gr(\Pi,n)$.
Since $\phi$ is a isometric immersion, it follows from the definitions of the product metrics on the $r$-fold products of the smooth manifolds $\Pj \Var{S}$ and $\Gr(\Pi, n)$, respectively, that the $r$-fold product
$$
\phi^{\times r}: (\mathbb{P}\Var S)^{\times r} \to \Gr(\Pi,n)^{\times r}, \;([\tensor A_1],\ldots,[\tensor A_r])\mapsto (\Tang{\tensor A_1}{\Var S},\ldots, \Tang{\tensor A_r}{\Var S})
$$
is an isometric immersion. The associated distance on $\Gr(\Pi,n)^{\times r}$ is $\dist_R$ from \cref{dist_R}. By \cref{appendix:isometry_importan_lemma} (3) this implies that
$$
\dist_{\mathrm{w}} \bigl( ([\tensor A_1],\ldots,[\tensor A_r]),\Sigma_\mathbb{P} \bigr) \ge \dist_\mathrm{R} \bigl( (\Tang{\tensor A_1}{\Var S},\ldots, \Tang{\tensor A_r}{\Var S}), \phi^{\times r}(\Sigma_\mathbb{P}) \bigr).
$$
Recall from \cref{GrSigma} the definition of $\Sigma_{\mathrm{Gr}}$ and note that $\phi^{\times r}(\Sigma_\mathbb{P})\subset \Sigma_\Gr$ by construction. Consequently,
$$
\dist_{\mathrm{w}} \bigl( ([\tensor A_1],\ldots,[\tensor A_r]),\Sigma_\mathbb{P} \bigr) \geq \dist_\mathrm{R}\bigl( (\Tang{\tensor A_1}{\Var S},\ldots, \Tang{\tensor A_r}{\Var S}), \Sigma_\Gr \bigr),
$$
so that, by \cref{part1:cnt_distances},
\begin{align*}
\dist_{\mathrm{w}} \bigl( ([\tensor{A}_1],\ldots,[\tensor{A}_r]),\Sigma_{\mathbb{P}} \bigr) &\geq \dist_\mathrm{P}\bigl( (\Tang{\tensor{A}_1}{\Var{S}},\ldots, \Tang{\tensor{A}_r}{\Var{S}}), \Sigma_\Gr \bigr).
\end{align*}
By \cref{characterization123}, the latter equals $\kappa(\tensor{A}_1,\ldots,\tensor{A}_r)^{-1}$, which proves the assertion.
\end{proof}


\section{Numerical experiments} \label{sec:experiments}
In this section, we perform a few numerical experiments in Matlab R2017b \cite{MATLAB2016} for illustrating \cref{cnt,thm_expected_value,cor_expected_value}.

\subsection{Distance to ill-posedness} \label{sec_numex_distance}
To illustrate \cref{cnt}, we performed the following experiment with tensors in $\R^{11}\otimes \R^{10} \otimes \R^5$. Note that the generic rank in that space is $23$. For each $2\leq r\leq 5$ we select an ill-posed tensor decomposition $A:=(\tensor A_1,\ldots, \tensor A_r)\in \Var S^{\times r}$ as explained next.
First, we sample a random rank-1 tuple $(\tensor A_1,\ldots,\tensor A_{r-1})$ in~$\R^{11 \times 10 \times 5}$. Suppose that $\tensor{A}_1 = \sten{a}{1}{1} \otimes \sten{a}{1}{2} \otimes \sten{a}{1}{3}$. Then, we take $\tensor A_r := \sten{a}{1}{1} \otimes \vect{x}_2 \otimes \vect{x}_3$,  where the components of~$\vect{x}_i$ are sampled from $N(0,1)$. Now,
\[
 \tensor{A}_1 + \tensor{A}_r = \sten{a}{i}{1} \otimes ( \sten{a}{i}{2} \otimes \sten{a}{i}{3} + \vect{x}_2 \otimes \vect{x}_3 ),
\]
and since a rank-$2$ matrix decomposition is never unique, it follows that $\tensor{A}_1 + \tensor{A}_r$ has at least a $2$-dimensional family\footnote{The fact that the family is at least two-dimensional follows from the fact that defect of the $2$-secant variety of the Segre embedding of $\R^m \times \R^n$ is exactly 2; see, e.g., \cite[Proposition 5.3.1.4]{Landsberg2012}.} of decompositions, and, hence, so does $\tensor{A}_1 + \cdots + \tensor{A}_r$. Then, it follows from \cite[Corollary 1.2]{BV2017} that $\kappa(A) = \infty$ and hence $A\in \Sigma_\Pj$. Finally, we generate a neighboring tensor decomposition $B:=(\tensor B_1,\ldots, \tensor B_r)\in \Var S^{\times r}$ by perturbing $A$ as follows. Let $\tensor{A}_i = \sten{a}{i}{1} \otimes \sten{a}{i}{2} \otimes \sten{a}{i}{3}$, and then we set $\tensor{B}_i = (\sten{a}{i}{1} + 10^{-2}\cdot\sten{x}{i}{1}) \otimes (\sten{a}{i}{2} + 10^{-2}\cdot\sten{x}{i}{2}) \otimes (\sten{a}{i}{3} + 10^{-2}\cdot\sten{x}{i}{3})$, where the elements of $\sten{x}{i}{k}$ are randomly drawn from $N(0,1)$.

Denote by $(0,1) \to \Var S^{\times r}, t\mapsto B_t$ a curve between $A$ and $B$ whose length is $\dist_\mathrm{w}(A,B)$. Then, for all~$t$, we have
  $\dist_\mathrm{w}(B_t,\Sigma_\Pj)\leq \dist_\mathrm{w}(A,B_t) $
and hence, by \cref{cnt},
  \begin{equation}\label{ineq_proxy}\frac{1}{\kappa(B_t)}\leq \dist_\mathrm{w}(A,B_t).\end{equation}
We expect for small $t$ that $\dist_\mathrm{w}(A,B_t)\approx \dist_\mathrm{w}(A,B_t)$ and so \cref{ineq_proxy} is a good substitute for the true inequality from \cref{cnt}.

The data points in the plots in \cref{figure_experiments} show, for each experiment, $\dist_\mathrm{w}(A,B_t)$ on the $x$-axis and~$\frac{1}{\kappa(B_t)}$ on the $y$-axis. Since all the data points are below the red line, it is clearly visible that \cref{ineq_proxy} holds. Moreover, since the data points (approximately) lie on a line parallel to the red line, the plots suggest, at least in the cases covered by the experiments, that for decompositions $A=(\tensor A_1,\ldots,\tensor A_r)$ close to~$\Sigma_\Pj$ the reverse of \cref{cnt} could hold as well, i.e.,
$
\dist_{\mathrm{w}}(([\tensor A_1],\ldots,[\tensor A_r]), \Sigma_\mathbb{P}) \le c \frac{1}{\kappa(\tensor A_1,\ldots,\tensor A_r)},
$ for some constant $c>0$ that might dependent on~$A$. For completeness, in the experiments shown in \cref{figure_experiments}, such a bound seems to hold for $c = 17$, $25$, $27$, $19$ respectively in the cases $r=2$, $3$, $4$, $5$.

\begin{figure}[tb]\footnotesize
\begin{center}
\includegraphics[width = 0.40\textwidth]{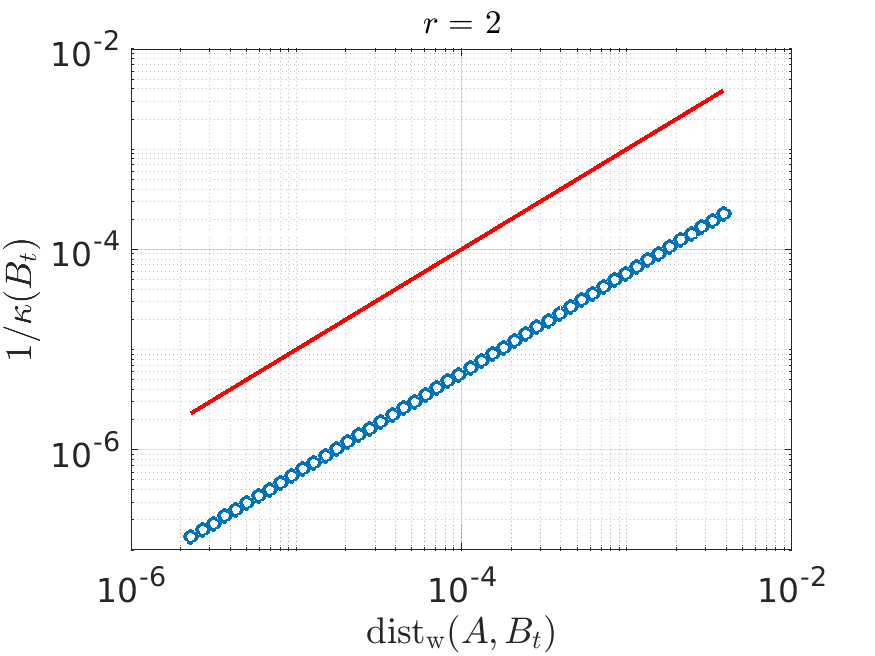}
\includegraphics[width = 0.40\textwidth]{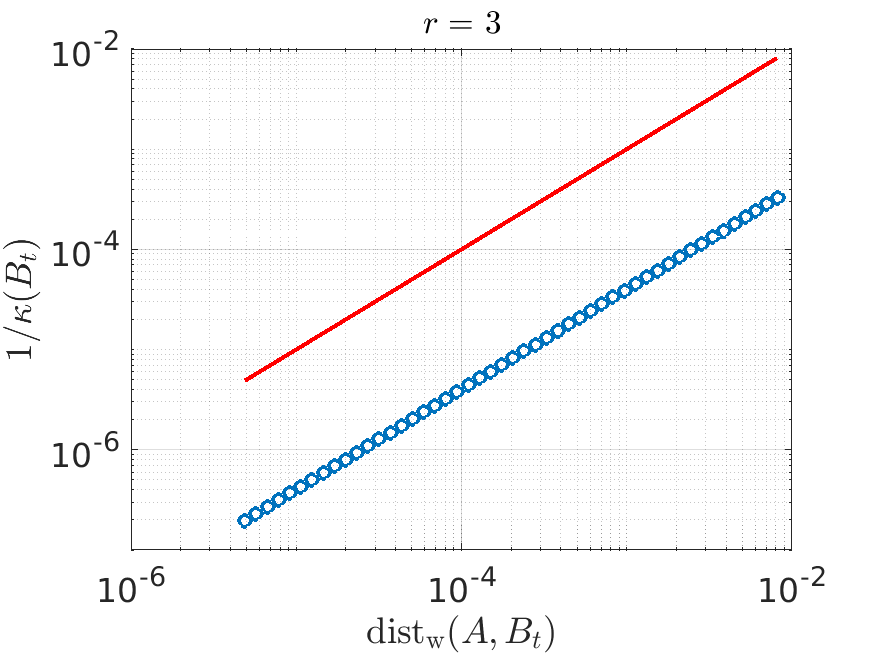}\vspace{0.8em}\\
\includegraphics[width = 0.40\textwidth]{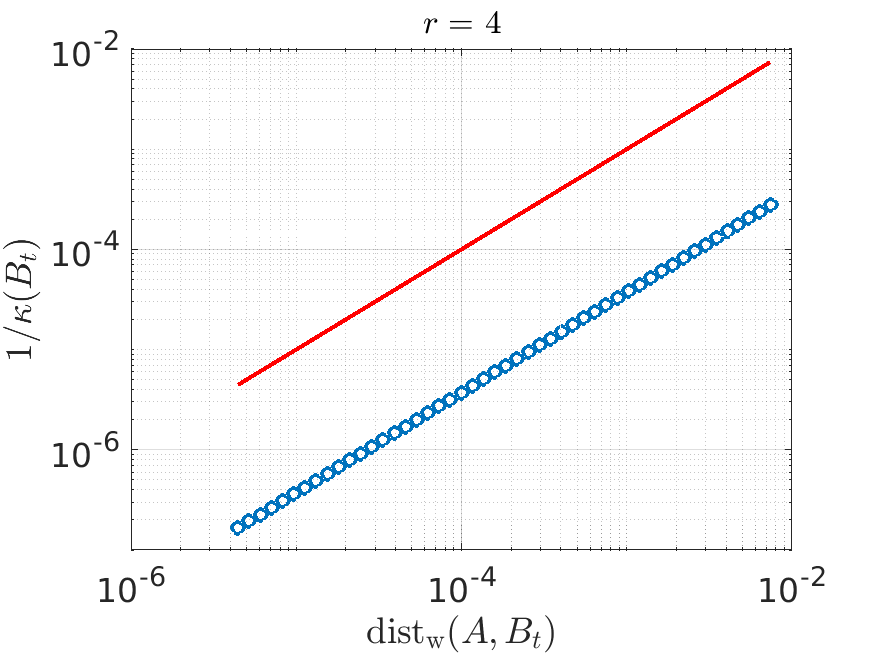}
\includegraphics[width = 0.40\textwidth]{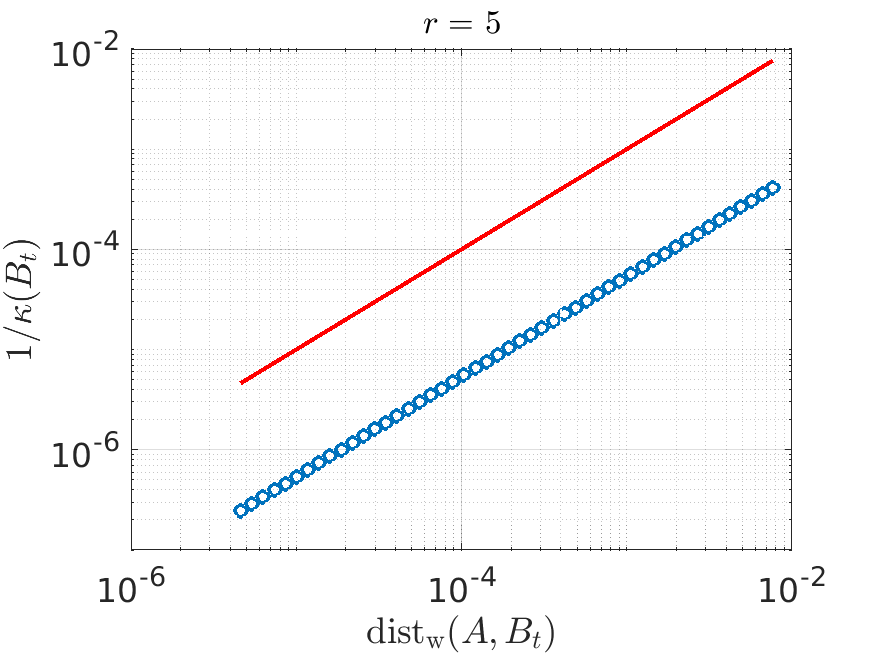}
\end{center}
\caption{\small
The blue data points compare the inverse condition number and the estimate of the weighted distance to the locus of ill-posed CPDs for the tensors described in \cref{sec:experiments}. The red line illustrates where the data points would lie if the inequality in \cref{cnt} were an equality. The gap between the red line and the blue data points is thus a measure for the sharpness of the bound in \cref{cnt}.
} \label{figure_experiments}
\end{figure}

\subsection{Distribution of the condition number} \label{sec_numex_distribution}

\begin{figure}[tb]
\begin{center}
\includegraphics[width = 0.90\textwidth]{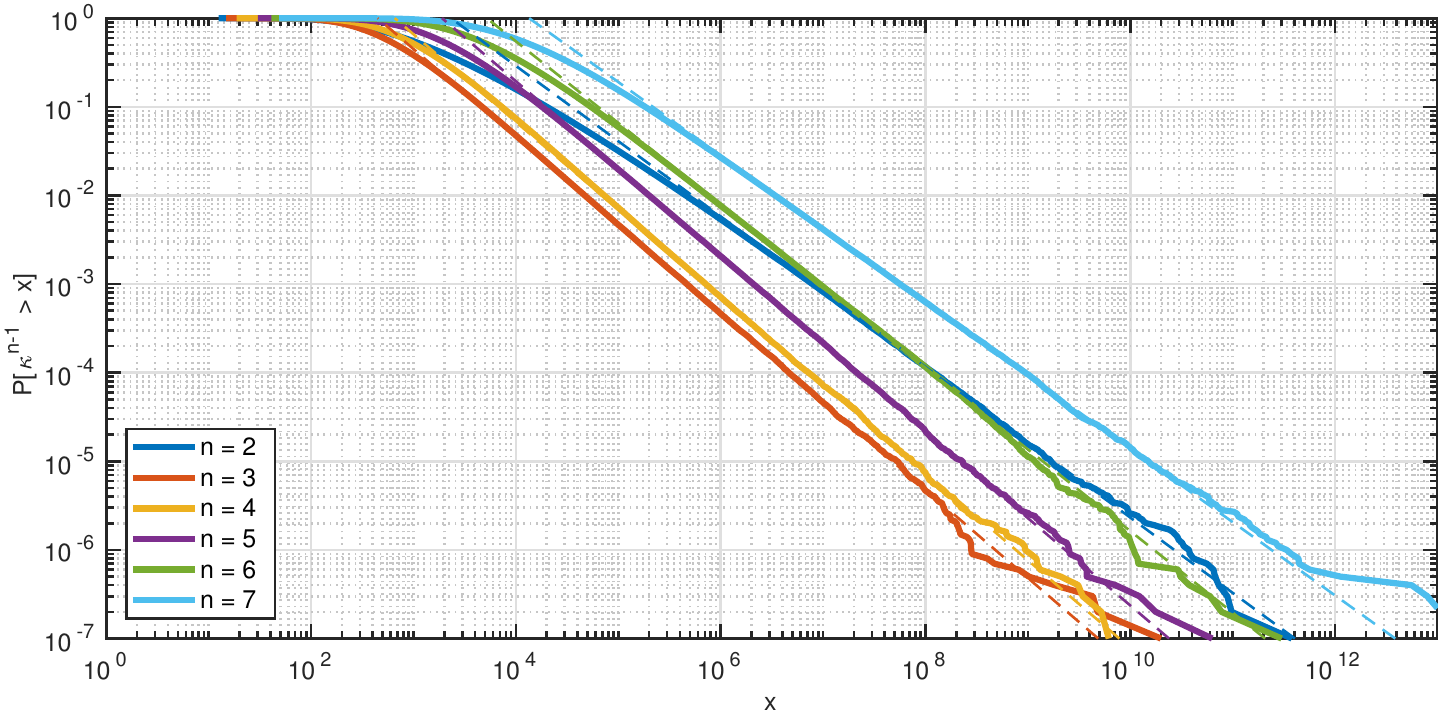}
\end{center}
\caption{\small
A log-log plot of the empirical complementary cumulative distribution function of the $(n-1)$th power of the condition number of random rank-$1$ tuples $(\tensor{A}_1, \ldots, \tensor{A}_7)$ in the space $\R^{7 \times 7 \times n}$ for $n = 2, 3, \ldots, 7$, computed from $10^7$ samples. The dashed lines represent approximations of the form $a_n x^{-b_n}$ of the empirical ccdf for $i=2,3,\ldots,7$; the parameters $(a_n,b_n)$ for each case are given in \cref{tab_parameters}.
}
\label{figure_powerdistribution}
\end{figure}

We perform Monte Carlo experiments for providing additional numerical evidence for \cref{thm_expected_value,cor_expected_value}. To this end, we randomly sampled $10^7$ random rank-$1$ tuples $(\tensor{A}_1,\ldots,\tensor{A}_7)$ in $\R^{7 \times 7 \times n}$, where $n=2,3,\ldots,7$, and computed their condition numbers. We will abbreviate the random variable $\kappa(\tensor{A}_1,\ldots,\tensor{A}_7)$ to $\kappa$ from now onwards.
These condition numbers are computed by constructing the~$49 n \times 7 (12 + n)$ block matrix $T = [U_i]_{i=1}^7$ from \cite[Theorem 1.1]{BV2017}, where the individual blocks $U_i$ are those from \cite[equation (5.1)]{BV2017}, and then computing the inverse of the least (i.e., the $7(12+n)$th) singular value of $T$. The outcome of this experiment is summarized in \cref{figure_powerdistribution}, where we plot the complementary cumulative distribution function (ccdf) of the $(n-1)$th power of the condition number; recall that we know from \cref{cor_expected_value} that $\mathbb{E}[ \kappa^{n-1} ] = \infty$.

It may appear at first glance that $\kappa^{n-1}$ behaves very erratically near the tails of the ccdfs in \cref{figure_powerdistribution}. This phenomenon is entirely due to the sample error. Indeed, as we took $10^7$ samples, this means that in the empirical ccdf, there are $10^k$ data points between $10^{-7} \le \mathrm{P}[ \kappa^{n-1} > x ] \le 10^{-7+k}$. For $k=1$ or $2$, the resulting sample error is visually evident.

It is particularly noteworthy that all of the ccdfs in \cref{figure_powerdistribution} roughly appear to be shifted by a constant; the slope of the curves looks rather similar. In the figure, there are additional dashed lines that appear to capture the asymptotic behavior of the ccdfs of $\kappa^{n-1}$ quite well. These straight lines in the log-log plot correspond to a hypothesized model $a_n x^{-b_n}$ with $a_n, b_n \ge 0$. In \cref{tab_parameters}, we give the (rounded) parameter values for these dashed lines in \cref{figure_powerdistribution}. By taking a log-transformation, fitting the model becomes a linear least squares problem, which was solved exactly. To avoid overfitting, we leave out the~$9.9 \cdot 10^6$ smallest condition numbers, that is, all data above the horizontal line $\mathrm{P}[\kappa^{n-1} > x] = 10^{-2}$, as well as the $100$ largest condition numbers, i.e., the data below the horizontal line $\mathrm{P}[\kappa^{n-1} > x] = 10^{-5}$. The motivation for this is as follows: the right tails of the ccdfs are corrupted by sampling errors, while for the left tails the model is clearly not valid. We are convinced that the hypothesized model is the correct one for very large condition numbers based on \cref{cnt}, which shows that a small distance from the ill-posed locus $\Sigma_\Pj$ the condition number grows at least like one over the distance, and the experiments from \cref{sec_numex_distance}, which show that close to the ill-posed locus the growth of the condition number appears also to be bounded by a constant times the inverse distance to $\Sigma_\Pj$. In other words, close to $\Sigma_\Pj$, the condition number behaves, as determined experimentally, asymptotically as $\kappa(A) = \mathcal{O}\bigl( (\dist_w(A,\Sigma_\Pj))^{-1} \bigr)$.

\begin{table} \small
\[
\begin{array}{lrrrrrr}
\toprule
n & 2 & 3 & 4 & 5 & 6 & 7 \\
\midrule
a_n & 2328.45 & 447.54  & 656.27   & 1902.08 & 5210.73 & 13485.19 \\
b_n & 1.17713 & 1.00514 & 1.01091  & 1.01415 & 1.08573 & 1.20828 \\
\midrule
R^2 & 0.99994 & 0.99987 & 0.99975 & 0.99988 & 0.99940 & 0.99972  \\
\bottomrule
\end{array}
\]
\caption{\small Parameters $(n,a_n,b_n)$ of the model $a_n x^{-b_n}$ fitted to the empirical cumulative distribution function described in \cref{figure_powerdistribution}. The row $R^2$ reports the coefficient of determination of the linear regression model $\log(a_n) - b_n \log( x )$ on the log-transformed empirical data; $R^2 = 1$ means the model perfectly predicts the data.}
\label{tab_parameters}
\end{table}

From the above discussion, we can conclude that for sufficiently large $x$, say $x \ge \kappa_0$, the true cdf of~$\kappa^{n-1}$, i.e.,
\(
F(x) = \mathrm{P}[\kappa^{n-1} \le x] = 1 - \mathrm{P}[ \kappa^{n-1} \ge x]
\)
is very well approximated by
\(
1 - a_n x^{-b_n} = \widetilde{F}(x).
\)
We can now employ the estimated cdfs to estimate the expected value of the $k$th power of the condition number $\kappa$ in the unknown cases $n=3,4,\ldots,7$ and $1 \le k \le n-2$. We are unable to compute these cases analytically because, firstly, we do not know whether the codimension of $\overline{\Sigma_\Pj}$ is one, and, secondly, the techniques in this paper can prove only lower bounds on the condition number.
We compute
\[
 \mathbb{E}[ \kappa^k ]
 = \mathbb{E}[ (\kappa^{n-1})^{\frac{k}{n-1}} ]
 = \int_0^\infty x^{\frac{k}{n-1}} \,\mathrm{d} F(x)
 = C + \int_{\kappa_0}^\infty x^{\frac{k}{n-1}} F'(x) \mathrm{d}x
 \approx C' + \int_{\kappa_0}^\infty x^{\frac{k}{n-1}} \widetilde{F}'(x) \mathrm{d}x,
\]
where in the last step we assume that the error term $E(x) = F'(x) - \widetilde{F}'(x)$ integrated against $x^{\frac{k}{n-1}}$ is at most a constant; this requires that the hypothesized model is asymptotically correct as $x \to \infty$, which seems reasonable based on the above experiments. So it follows that
\[
 \mathbb{E}[ \kappa^k ] \approx C' + \int_{\kappa_0}^\infty  a_n b_n x^{-b_n - 1 + \frac{k}{n-1}}  \mathrm{d}x.
\]
Note that the critical value for obtaining a finite integral is $k < (n-1) b_n$. Incidentally, the integral computed from the hypothesized model is finite for $n=2$, as $1 < 1.17713$, but we attribute this $17\%$ error of $b_n$ to the sample variance, as we have proved in \cref{cor_expected_value} that the true integral is infinity.
For $n \ge 3$, all of the hypothesized integrals with $1 \le k \le n-2$ integrate to constants; the computed values~$b_n$ would have to be off by $27\%$ before the case $n=5$ with $k=3$ integrates to infinity. This provides some indications that the expected value of the condition number $\kappa$ will be \emph{finite} for $n_1 \times n_2 \times n_3$ tensors, provided that all $n_i \ge 3$. It is therefore unlikely that \cref{cor_expected_value} may be improved by the techniques considered in this paper.

\section{Conclusions} \label{sec_conclusions}
We presented a technique for establishing whether the average condition number of CPDs is infinite, namely \cref{thm_expected_value}. This is based on the partial condition number theorem, \cref{cnt}, that bounds the inverse condition number by a distance to the locus of ill-posed CPDs. Using this strategy, we showed that the average of powers of the condition numbers of random rank-$1$ tuples of length $r$ can be infinite in \cref{cor_expected_value}, depending on the codimension of the ill-posed locus. In particular, it was proved that the average condition number for $n_1 \times n_2 \times 2$ tensors is infinite. We are convinced that the inability to reduce the power in \cref{cor_expected_value} to $1$ for $n_1 \times n_2 \times n_3$ tensors with $2 \le n_1, n_2, n_3 \le 10$, as shown in \cref{prop_impossible}, along with the numerical experiments in \cref{sec_numex_distribution}, are a strong indication that the average condition number is finite for tensors for which $n_1, n_2, n_3 \ge 3$.

The large gap in sensitivity between the case of $n_1 \times n_2 \times 2$ tensors and larger tensors has negative implications for the numerical stability of algorithms for computing CPDs based on a generalized eigendecomposition \cite[such as those by][]{LRA1993,Lorber1985,SK1990,SY1980}, as is shown by \cite{BBV2018}.

The strategy presented in this article cannot prove that the average condition number is finite. However, we believe that the main components of our approach can be adapted to prove upper bounds on the average condition number, provided that one can establish a local converse to \cref{cnt}.

\appendix
\section{Proof of \cref{isometry_thn}} \label{sec_proof_isometry}

In this section we prove \cref{isometry_thn} to complete our study. We abbreviate $\mathbb{P}^{m-1}:=\mathbb{P}(\R^m)$ in the following. Consider the following commutative diagram:
$$ 
\begin{tikzcd}
\Pj^{n_1-1} \times \cdots \times \Pj^{n_d-1} \arrow{r}{\sigma} \arrow{d}[left]{\psi:=\iota \circ \phi \circ \sigma} & \arrow{d}{\phi} \mathbb{P}\Var{S} \\
\Pj( \wedge^n \R^\Pi ) & \Gr(\Pi,n) \arrow[hookrightarrow]{l}{\iota}
\end{tikzcd}
$$
Herein, $\sigma$ as defined in \cref{eqn_def_sigma} is an isometry by the definition, $\phi$ is defined as in the statement of the proposition, and $\iota$ is the Pl\"ucker embedding \cite[Chapter~3.1.]{gkz}, which maps into the space of \emph{alternating tensors} $\Pj( \wedge^n \R^\Pi )$. Recall from \cite[Section 2.6]{Landsberg2012} that alternating tensors are linear combinations of alternating rank-1 tensors like
\[ \vect{x}_1\wedge \cdots \wedge \vect{x}_d := \frac{1}{d!} \sum_{\pi \in \mathfrak{S}_d} \operatorname{sgn}(\pi) \vect{x}_{\pi_1} \otimes \vect{x}_{\pi_2} \otimes \cdots \otimes \vect{x}_{\pi_d};
\]
where $\mathfrak{S}_d$ is the permutation group on $\{1, \ldots, d\}$.

The image of the Pl\"ucker embedding $\mathscr{P}:=\iota(\Gr(\Pi,n))\subset \mathbb{P}\left(\wedge^{n}\R^{\Pi}\right)$ is a smooth variety called the \emph{Pl\"ucker variety}.
The Fubini--Study metric on $\mathbb{P}\left(\wedge^{n}\R^\Pi\right)$ makes $\mathscr{P}$ a Riemannian manifold. The Pl\"ucker embedding is an isometry; see, e.g., \cite[Section 2]{Griffiths1974} or \cite[Chapter 3, Section 1.3]{Fuchs2004}.

Since $\sigma$ and $\iota$ are isometries, it follows from \cref{appendix:isometry_importan_lemma} that $\phi$ is an isometric immersion if and only if $\psi := \iota \circ \phi \circ \sigma$ is an isometric immersion. We proceed by proving the latter.
According to \cref{appendix:isometry_dfn}, we have to prove that for all $p\in \mathbb{P}^{n_1-1}\times \cdots \times \mathbb{P}^{n_d-1}$ and for all $x, y\in \Tang{p}{(\mathbb{P}^{n_1-1}\times \cdots \times \mathbb{P}^{n_d-1})}$ we have
\begin{equation*}
\langle x, y\rangle_{\text{w}} = \langle (\deriv{\psi}{p})(x), (\deriv{\psi}{p})(y) \rangle.
\end{equation*}
However, the equality $2\langle x, y \rangle =\langle  x- y, x- y\rangle -  \langle  x, x\rangle - \langle  y, y \rangle$ shows that it suffices to prove
  \begin{equation}\label{cnt2}\forall p\in \mathbb{P}^{n_1-1}\times \cdots \times \mathbb{P}^{n_d-1}: \forall x\in \Tang{p}{(\mathbb{P}^{n_1-1}\times \cdots \times \mathbb{P}^{n_d-1})}:
\langle x,x\rangle_{\text{w}} = \langle (\deriv{\psi}{p})(x), (\deriv{\psi}{p})(x)\rangle.
 \end{equation}
To show this, let $p\in \mathbb{P}^{n_1-1}\times \cdots \times \mathbb{P}^{n_d-1}$ and $x\in \Tang{p}{(\mathbb{P}^{n_1-1}\times \cdots \times \mathbb{P}^{n_d-1})}$ be fixed and consider any smooth curve $\gamma : (-1,1) \to \Pj^{n_1-1} \times \cdots \times \Pj^{n_d-1}$ with $\gamma(0) = p$ and $\gamma'(0) = x$. The action of the differential is computed as follows according to \cite[Corollary 3.25]{Lee2013}:
\[
 (\deriv{\psi}{p})(x) = \deriv{(\psi \circ \gamma)}{0}.
\]
We compute the right-hand side of that equation. However, before taking derivatives, we first compute an expression for $(\psi\circ\gamma)(t)$.

Because $
\Tang{p}{(\Pj^{n_1-1} \times \cdots \times \Pj^{n_d-1})} = \Tang{p_1}{\Pj^{n_1-1}} \times \cdots \times \Tang{p_d}{\Pj^{n_d-1}},
$
we can write $x = (x_1, \ldots, x_d)$ with $x_i \in \Tang{p_i}{\Pj^{n_i-1}}$.
For each $i$, we denote by $\vect{a}_i \in \mathbb{S}(\R^{n_i})$ a unit-norm representative for $p_i$, i.e., $p_i= [\vect{a}_i]$ with~$\| \vect{a}_i \|~=~1$ in the Euclidean norm. Letting $\vect{a}_i^\perp = \cset{\vect{u} \in \R^{n_i}}{\langle \vect{u}, \vect{a}_i \rangle = 0}$
denote the orthogonal complement of $\vect{a}_i$ in $\R^{n_i}$, we can then identify~$\vect{a}_i^\perp=\Tang{p_i}{\Pj^{n_i-1}}$ by \cref{tangent_proj}. Moreover, because $\vect{a}_i$ is of unit norm, the Fubini--Study metric on $\Tang{p_i}{\Pj^{n_i-1}}$ is given by the Euclidean inner product on the linear subspace $\vect{a}_i^\perp$.
Now, let~$\vect{x}_i$ denote the unique vector in~$\vect{a}_i^\perp$ corresponding to~$x_i$. The sphere $\mathbb{S}(\R^{n_i})$ is a smooth manifold, so we find a curve $\gamma_i : (-1,1) \to \mathbb{S}(\R^{n_i})$ with $\gamma_i(0) = \vect{a}_i$ and $\gamma_i'(0) = \vect{x}_i$. Without loss of generality we assume that $\gamma_i$ is the exponential map \cite[Chapter 20]{Lee2013}.
We claim that we can write $\gamma$ as $\gamma(t) = (\pi_1 \circ \gamma_1(t), \ldots, \pi_d \circ \gamma_d(t))$, where $\pi_i : \mathbb{S}(\R^{n_i}) \to \Pj^{n_i-1}$ is the canonical projection. Indeed, we have $\gamma(0) = ([\vect{a}_1], \ldots, [\vect{a}_d]) = p$ and
\begin{align*}
\gamma'(0)
= \bigl( (\pi_1 \circ \gamma_1)'(0), \ldots, (\pi_d \circ \gamma_d)'(0) \bigr)
&= \bigl( \mathrm{P}_{(\vect{a}_1^\perp)} \gamma_1'(0), \ldots, \mathrm{P}_{(\vect{a}_d^\perp)} \gamma_d'(0) \bigr)\\
&= \bigl( \mathrm{P}_{(\vect{a}_1^\perp)} \vect{x}_1, \ldots, \mathrm{P}_{(\vect{a}_d^\perp)} \vect{x}_d \bigr)
= \bigl( \vect{x}_1, \ldots, \vect{x}_d \bigr) = x,
\end{align*}
where $\mathrm{P}_A$ denotes the orthogonal projection onto the linear space $A$, where the second equality is due to \cite[Lemma 14.8]{condition}, and where the last step is due to the identification $\vect{a}_i^\perp \simeq \Tang{p_i}{\Pj^{n_i-1}}$. This shows
$(\psi\circ\gamma)(t) = \psi(\pi_1 \circ \gamma_1(t), \ldots,   \pi_d\circ \gamma_d(t))$.
Recall that $\psi = \iota \circ\phi\circ \sigma$ and that
\[
 (\phi\circ\sigma \circ \gamma)(t) = \Tang{\gamma_1(t) \otimes \cdots \otimes \gamma_d(t)}\,\Var S.
\]
Hence, $(\psi\circ\gamma)(t) = \psi(\Tang{\gamma_1(t) \otimes \cdots \otimes \gamma_d(t)}\,\Var S)$. To compute the latter we must give a basis for the tangent space $\Tang{\gamma_1(t) \otimes \cdots \otimes \gamma_d(t)}\,\Var S$. To do so, let us denote by $\{\vect{u}_{1}^i(t),\vect{u}_{2}^i(t), \ldots, \vect{u}_{n_i-1}^i(t)\}$ an orthonormal basis for the orthogonal complement of $\gamma_i(t)$; such a moving orthonormal basis is called an \emph{orthonormal frame}.
%
%
%
%
Then, by \cite[Section 4.6.2]{Landsberg2012} a basis for~$\Tang{\gamma_1(t) \otimes \cdots \otimes \gamma_d(t)}{\Var{S}}$ is given by
\[
\Var{B}(t) = \set{\tensor{A}(t)} \cup \cset{\tensor{A}_{(i,j)}(t)}{1\leq i\leq d, 1\leq j \leq n_i-1},
\]
where
 \begin{align}\label{f_ij2}
 \tensor{A}(t) :=& \gamma_1(t) \otimes \cdots \otimes \gamma_d(t)
 \;\text{ and }\;\\
 \tensor A_{(i,j)}(t) =&
 \gamma_1(t) \otimes\cdots\otimes \gamma_{i-1}(t) \otimes \vect{u}_j^i(t) \otimes \gamma_{i+1}(t) \otimes\cdots\otimes \gamma_d(t).\nonumber
 \end{align}
 If we let $\pi$ denote the canonical projection $\pi:\wedge^{n} \R^\Pi \to \mathbb{P}\left(\wedge^{n} \R^\Pi \right)$, then we find
\begin{equation}\label{cnt1}
(\psi \circ \gamma)(t)
= \iota(\mathrm{span}\,\mathcal{B}(t))
= \pi \left(\tensor{A}(t) \wedge \Biggl( \bigwedge_{i=1}^d \bigwedge_{j=1}^{n_i-1} \tensor{A}_{(i,j)}(t) \Biggr) \right);
\end{equation}
see \cite[Chapter~3.1.C]{gkz}. Note in particular that the right-hand side of \cref{cnt1} is independent of the specific choice of the orthonormal bases $\mathcal{B}(t)$, because the exterior product of another basis is just a scalar multiple of the basis we chose (below we make a specific choice of $\mathcal{B}(t)$ that simplifies subsequent computations). In the following let
$$\vect{g}(t): = \tensor{A}(t) \wedge \Biggl( \bigwedge_{i=1}^d \bigwedge_{j=1}^{n_i-1} \tensor{A}_{(i,j)}(t) \Biggr).$$
We are now prepared to compute the derivative of $(\psi \circ \gamma)(t) = (\pi \circ \vect{g})(t) = [\vect{g}(t)]$. According to \cite[Lemma 14.8]{condition}, we have
\[
\deriv{(\psi \circ \gamma)}{0} = \mathrm{P}_{(\vect{g}(0))^\perp} \frac{\vect{g}'(0)}{\|\vect{g}(0)\|}.
\]
We will first prove that $\| \vect{g}(t) \| = 1$, which entails that
$\vect{g}(t) \subset \mathbb{S}(\wedge^n \R^\Pi)$ so that
\[
\deriv{(\psi \circ \gamma)}{0} = \mathrm{P}_{(\vect{g}(0))^\perp} \vect{g}'(0) = \vect{g}'(0) = \deriv{\vect{g}}{0},
\]
as $\vect{g}'(t)$ would in this case be contained in the tangent space to the sphere over $\wedge^n \R^\Pi$. We now need the following standard result.

\begin{lemma} \label{inner_droducts}
We have the following:
  \begin{enumerate}\item
For $1\leq k\leq d$, let $\vect{x}_k, \vect{y}_k\in \R^{n_k}$, and let $\langle \cdot, \cdot \rangle$ denote the standard Euclidean inner product. Then, the inner product of rank-1 tensors satisfies $\langle \vect{x}_1 \otimes \cdots \otimes \vect{x}_d, \, \vect{y}_1 \otimes \cdots \otimes \vect{y}_d \rangle = \prod_{j=1}^d \langle \vect{x}_j, \vect{y}_j\rangle$.
\item Let $\vect{x}_1,\ldots,\vect{x}_d,\vect{y}_1,\ldots,\vect{y}_d\in\R^m$. Let $\langle \cdot, \cdot \rangle$ be the standard Euclidean inner product. Then, the inner product of skew-symmetric rank-1 tensors satisfies $\langle \vect{x}_1 \wedge \cdots \wedge \vect{x}_d, \, \vect{y}_1 \wedge \cdots \wedge \vect{y}_d \rangle = \det\bigl( [ \langle \vect{x}_i,\vect{y}_j\rangle ]_{i,j=1}^d \bigr).$
\item Whenever $\{ \vect{x}_1, \ldots, \vect{x}_d \}$ is a linearly dependent set, we have $\vect{x}_1 \wedge \cdots \wedge \vect{x}_d = 0.$
\end{enumerate}
\end{lemma}
\begin{proof}
For the first point see, e.g., \cite[Section 4.5]{Hackbusch2012}. For the second see, e.g., \cite[Section 4.8]{Greub1978} or \cite[Proposition 14.11]{Lee2013}. The third is a consequence of the second point.
\end{proof}

Using the computation rules for inner products from \cref{inner_droducts} we find
\begin{align}
\label{a1} \langle \tensor{A}(t), \tensor{A}(t) \rangle &= \prod_{i=1}^d \langle \gamma_i(t), \gamma_i(t) \rangle =1;\\
\label{a2} \langle \tensor{A}(t), \tensor{A}_{(i,j)}(t) \rangle  &= \langle \gamma_i(t), \vect{u}_j^i(t) \rangle \prod_{k\neq i} \langle \gamma_k(t), \gamma_k(t) \rangle =0;\\
\label{a3} \langle \tensor{A}_{(i,j)}(t), \tensor{A}_{(k,\ell)}(t) \rangle &= \begin{cases} 1,& \text{if } (i,j)=(k,\ell), \\0,&\text{else.}\end{cases}
\end{align}
In other words, $\Var{B}(t)$ is an \emph{orthonormal basis} for $\Tang{\tensor{A}(t)}{\Var S}= \Tang{\gamma_1(t) \otimes \cdots \otimes \gamma_d(t)}{\Var S}$. By \cref{inner_droducts}, we have
$$
\langle \vect{g}(t),\vect{g}(t) \rangle
= \det
\begin{bmatrix}
\langle \tensor{A}(t),\tensor{A}(t) \rangle & \langle \tensor{A}(t), \tensor{A}_{(1,1)}(t) \rangle & \cdots & \langle \tensor{A}(t),\tensor{A}_{(d,n_d)}(t)\rangle \\
\langle \tensor{A}_{(1,1)}(t),\tensor{A}(t)\rangle & \langle \tensor{A}_{(1,1)}(t),\tensor{A}_{(1,1)}(t)\rangle & \cdots & \langle \tensor{A}_{(1,1)}(t),\tensor{A}_{(d,n_d)}(t)\rangle \\
\vdots & \vdots  & \ddots & \vdots \\
\langle \tensor{A}_{(d,n_d)}(t),\tensor{A}(t)\rangle & \langle \tensor{A}_{(d,n_d)}(t),\tensor{A}_{(1,1)}(t)\rangle & \cdots & \langle \tensor{A}_{(d,n_d)}(t),\tensor{A}_{(d,n_d)}(t)\rangle
\end{bmatrix},
$$
which equals $\det I_n = 1$.

It now only remains to compute $\deriv{\vect{g}}{0}$. For this we have the following result.

\begin{lemma}\label{lem_help}
Let $\tensor{A} := \tensor{A}(0)$ and $\tensor{A}_{(i,j)} := \tensor{A}_{(i,j)}(0)$ and write
$$\vect{f}_{(i,j)} := \tensor A\wedge \tensor A_{(1,1)}\wedge \cdots \wedge \tensor A_{(i,j-1)}\wedge \tensor{A}_{(i,j)}'(0) \wedge \tensor{A}_{(i,j+1)} \wedge \cdots\wedge \tensor A_{(p,n_d-1)}. $$
The differential satisfies
$
\deriv{\vect g}{0} = \sum_{i=1}^d\sum_{j=1}^{n_i-1} \vect{f}_{(i,j)},
$
where
\(
\left\langle \vect f_{(i,j)}, \vect f_{(k,\ell)}\right\rangle = \delta_{ik}\delta_{j\ell} \sum_{1 \le \lambda \ne i \le d} \langle \vect x_\lambda, \vect x_\lambda \rangle,
\)
where $\delta_{ij}$ is the Kronecker delta.
\end{lemma}
We prove this lemma at the end of this section. We can now prove \cref{cnt2}. From \cref{lem_help}, we find
\begin{equation*}
\langle (\deriv{\psi}{p})(x), (\deriv{\psi}{p})(x) \rangle
= \langle \deriv{\vect{g}}{0}, \deriv{\vect{g}}{0} \rangle
= \left\langle \sum_{i=1}^d \sum_{j=1}^{n_i-1} \vect{f}_{(i,j)}, \sum_{k=1}^{d} \sum_{\ell=1}^{n_k-1} \vect{f}_{(k,\ell)} \right\rangle
= \sum_{i=1}^d \sum_{j=1}^{n_i -1} \sum_{1 \le \lambda \ne i \le d} \langle \vect{x}_\lambda, \vect{x}_\lambda \rangle.
\end{equation*}
Reordering the terms, one finds
$$
\langle (\deriv{\psi}{p})(x), (\deriv{\psi}{p})(x) \rangle
= \sum_{i=1}^d \langle \vect{x}_i,\vect{x}_i \rangle \sum_{1\le \lambda \ne i \le d} \sum_{j=1}^{n_\lambda - 1} 1
= \sum_{i=1}^d \langle \vect{x}_i, \vect{x}_i \rangle \cdot (n -n_i)
= \langle \vect x, \vect x \rangle_{\mathrm{w}},
$$
where the penultimate equality follows from the formula $n = 1 + \sum_{i=1}^d (n_i-1)$ in \cref{dim_segre}.
This proves \cref{cnt2} so that $\phi$ is an isometric map.

Finally, \cref{cnt2} also entails that $\phi$ is an immersion. Indeed, for an immersion it is required that $\deriv{\psi}{p}$ is injective. Suppose that this is false, then there is a nonzero $x \in \Tang{p}{(\Pj^{n_1-1} \times \cdots \times \Pj^{n_d-1})}$ with corresponding nonzero $\vect{x}$ such that
\[
0 = \langle 0, 0 \rangle = \langle (\deriv{\psi}{p})(x), (\deriv{\psi}{p})(x) \rangle = \langle \vect{x}, \vect{x} \rangle_{\mathrm{w}} > 0,
\]
which is a contraction. Consequently, $\phi$ is an isometric immersion, concluding the proof.\qed

~

It remains to prove \cref{lem_help}.

\begin{proof}[Proof of \cref{lem_help}]
Recall that we have put $\vect{a}_i:=\gamma_i(0) \in\mathbb{S}(\R^{n_i})$ and $ \vect{x}_i:=\gamma_i'(0)\in \Tang{\vect a_i}{\mathbb{S}(\R^{n_i})}$ for $1\leq i\leq d$. Without restriction we can assume that $\gamma_i$ is contained in the great circle through $\vect a_i$ and~$\vect x_i$. As argued above, we have the freedom of choice of an orthonormal basis of each~$\gamma_i(t)^\perp$. To simplify computations we make the following choice. %

For all $i$, let $\vect{u}_2^i, \ldots, \vect u_{n_i-1}^i$ be an orthonormal basis for $\vect a_i^\perp \cap \,\vect x_i^\perp$ and consider the orthogonal transformation
$
U
$
that rotates $\vect{a}_i$ to~$\Vert\vect x_i\Vert^{-1}\vect{x}_i$, $\vect{x}_i$ to $-\Vert\vect x_i\Vert\vect{a}_i$ and leaves $\set{\vect u_{2}^i,\ldots,\vect u_{n-1}^i}$ fixed.
Then,  we define the following curves (which expect for the first one are all constant).
$$
\vect u_1^i(t) := U\gamma_i(t),\quad \vect u_2^i(t) :=\vect{u}_2^i,\quad\ldots\quad \vect u_{n_i-1}^i(t) :=\vect{u}_{n_i-1}^i.
$$
By construction $\{\vect{u}_1^i(t), \vect{u}_2^i(t), \ldots, \vect u_{n-1}^i(t)\}$ is an orthonormal basis for the orthogonal complement of $\gamma_i(t)$ for all $t$. We have
\begin{equation}\label{type_of_derivatives}
\deriv{\vect{u}_1^i(t)}{0} = U\gamma_i'(0) = -\Vert\vect x_i\Vert\vect a_i,\quad \deriv{\vect{u}_2^i(t)}{0} = \cdots= \deriv{\vect{u}_{n_i-1}^i(t)}{0}  =0.
\end{equation}
\begin{figure}[tb]\small
\begin{center}
  \begin{tikzpicture}[scale=1,thick]
    \draw[] (0,3) node {};

    \draw[black,dashed,color=blue]  (0,2) arc (90:115:2) (0,2) arc (90:65:2);
    \draw[black,dashed,color=blue]  (2,0) arc (0:25:2) (2,0) arc (0:-25:2);
    \draw[black,->] (0,0) --node[left] {$\frac{\vect{x}_i}{\Vert \vect x_i\Vert}$} (0,2);
    \draw[black,->] (0,0) --node[above] {$\vect{a}_i$} (2,0);
    \draw[black,->] (0,0) --node[above left] {$\vect{u}_2^i$} (-1,-1);

    \draw[black,->] (2,0) --node[right] {$\gamma_i'(0)$} (2,.87);
    \draw[black,->] (0,2) --node[above] {$U\gamma_i'(0)$} (-.87,2);

    \draw[] (1.3,2.05) node[color=blue] {$U\gamma_i(t)$};
    \draw[] (1.45,-.7) node[color=blue] {$\gamma_i(t)$};
  \end{tikzpicture}
\caption{A sketch of the orthonormal frame $\{\gamma_i(t), U\gamma_i(t), \vect{u}_2^i(t), \ldots, \vect u_{n_i-1}^i(t)\}$.}
\label{fig1}
\end{center}
\end{figure}

We will use this choice of orthonormal bases for the remainder of the proof.
By the definition of $\vect{g}(t)$ and the product rule of differentiation, the first term of $\deriv{\vect{g}}{0}$ is $\tensor{A}'(0) \wedge \bigwedge_{i=1}^d \bigwedge_{j=1}^{n_i-1} \tensor{A}_{(i,j)}$.
We have
\begin{equation}\label{f_deriv}
\tensor{A}'(0) = \sum_{\lambda=1}^d \vect a_1\otimes \cdots \otimes \vect a_{\lambda-1}\otimes \vect x_\lambda\otimes \vect a_{\lambda+1}\otimes\cdots\otimes \vect a_d = \sum_{\lambda=1}^d \Norm{\vect x_\lambda} \tensor{A}_{(\lambda,1)}.
\end{equation}
Hence, from the multilinearity of the exterior product it follows that the first term of $\deriv{\vect{g}}{0}$ is
\begin{equation*}
\sum_{\lambda=1}^d \Norm{\vect x^\lambda} \left( \tensor A_{(\lambda,1)} \wedge \tensor{A}_{(1,1)} \wedge \cdots \wedge \tensor{A}_{(d,n_d-1)} \right) = \sum_\lambda 0 = 0.
\end{equation*}
This implies that all of the terms of $\deriv{\vect{g}}{0}$ involve $\tensor{A}_{(i,j)}'(0)$ for some $(i,j)$. From~\cref{f_ij2}, we find
$$
\tensor{A}_{(i,j)}'(0)  =  \sum_{\lambda = 1}^d \tensor{A}_{(i,j)}^\lambda,
$$
where, using the shorthand notation $\vect{u}_j^i = \vect{u}_j^i(0)$, we have put
\begin{equation*}
\tensor{A}_{(i,j)}^\lambda :=
\begin{cases}
\vect{a}_1 \otimes \cdots \otimes \vect{a}_{\lambda-1} \otimes \vect{x}_\lambda \otimes \vect{a}_{\lambda+1} \otimes \cdots \otimes \vect{a}_{i-1} \otimes \vect{u}_j^i \otimes \vect{a}_{i+1} \otimes \cdots \otimes \vect{a}_d & \text{if } \lambda \ne i, \\
\vect a^1\otimes\cdots\otimes \vect a^{i-1} \otimes \deriv{\vect{u}_j^i(t)}{0}\otimes \vect a^{i+1}\otimes \cdots  \otimes \vect a^d, & \text{otherwise}.
\end{cases}
\end{equation*}
Recall from \refeqn{type_of_derivatives} that
$\deriv{\vect{u}_1^i(t)}{0}  = -\|\vect{x}_i\| \vect{a}_i$, while for $j>1$ we have $\deriv{\vect{u}_j^i(t)}{0} = 0$. Hence,
\begin{equation*}
\tensor{A}_{(i,j)}^\lambda :=
\begin{cases}
\vect{a}_1 \otimes \cdots \otimes \vect{a}_{\lambda-1} \otimes \vect{x}_\lambda \otimes \vect{a}_{\lambda+1} \otimes \cdots \otimes \vect{a}_{i-1} \otimes \vect{u}_j^i \otimes \vect{a}_{i+1} \otimes \cdots \otimes \vect{a}_d & \text{if } \lambda \ne i, \\
\vect a^1\otimes\cdots\otimes \vect a^{i-1} \otimes (-\|\vect{x}_i\| \vect{a}_i) \otimes \vect a^{i+1}\otimes \cdots  \otimes \vect a^d, &\text{if } (\lambda,j) = (i,1), \\
0 & \text{otherwise}.
\end{cases}
\end{equation*}
Then,
\begin{align}
\label{eqn_fij_def} \vect{f}_{(i,j)}
 &= s_{(i,j)}\, \tensor{A} \wedge \left( \sum_{\lambda=1}^d \tensor{A}_{(i,j)}^\lambda \right) \wedge \bigwedge_{i=1}^d \bigwedge_{1 \le j \ne i < n_i} \tensor{A}_{(i,j)} \\
\nonumber &= s_{(i,j)} \sum_{1 \le \lambda \ne i \le d} \tensor{A} \wedge \tensor{A}_{(i,j)}^\lambda \wedge \bigwedge_{i=1}^d \bigwedge_{1 \le j \ne i < n_i} \tensor{A}_{(i,j)}
 =: s_{(i,j)} \sum_{1 \le \lambda \ne i \le d} \vect{f}_{(i,j)}^\lambda,
\end{align}
where $s_{(i,j)} \in \{ -1,1\}$ is the sign of the permutation for moving $\tensor{A}_{(i,j)}'(0)$ to the second position in the exterior product. We continue by computing for $\lambda\ne i$ and $\mu\ne k$ the value
\begin{align*}
\langle \vect{f}_{(i,j)}^\lambda, \vect{f}_{(k,\ell)}^\mu \rangle = \det \bigl( B_{(i,j),\lambda}^T B_{(k,\ell),\mu} \bigr), \text{ where }\;
B_{(i,j),\lambda} :=
\begin{bmatrix}
  \tensor{A} & \tensor{A}_{(i,j)}^\lambda & [[\tensor{A}_{(i,j)}]_{j\ne i} ]_{i=1}^d
\end{bmatrix};
\end{align*}
herein, the column vectors should be interpreted as vectorized tensors. Recall that $\langle \vect{a}_i, \vect{x}_i \rangle = 0$ and that~$\langle \vect{a}_i, \vect{u}_j^i \rangle = 0$ for all $i,j$. Then, it follows from \cref{inner_droducts} and direct computations that for $\lambda\ne i$ and $\mu\ne k$, we have
\begin{align*}
\langle \tensor{A}, \tensor{A}_{(k,\ell)}^\mu \rangle &= \langle \tensor{A}, \tensor{A}_{(k,\ell)} \rangle = 0,\quad
\langle \tensor{A}_{(i,j)}^\lambda, \tensor{A}_{(k,\ell)}^\mu \rangle = \delta_{ik} \delta_{j\ell} \delta_{\lambda\mu} \Vert \vect x_\lambda \Vert^2, \quad \text{and}\quad
\langle \tensor{A}_{(i,j)}^\lambda, \tensor A_{(k,\ell)} \rangle =0.
\end{align*}
We distinguish between two cases.
If $(i,j)\neq (k,\ell)$, $\lambda\ne i$ and $\mu\ne k$, it follows from the above equations that the row of $( B_{(i,j),\lambda})^{T} B_{(k,\ell),\mu}$ consisting of
\[
 \begin{bmatrix} \langle \tensor{A}_{(i,j)}^\lambda, \tensor{A} \rangle & \langle \tensor{A}_{(i,j)}^\lambda, \tensor{A}_{(k,\ell)}^\mu \rangle & [ [\langle \tensor{A}_{(i,j)}^\lambda, \tensor{A}_{(k,\ell)} \rangle ]_{\ell\ne k} ]_k \end{bmatrix}
\]
is a zero row, which implies that $\left\langle \vect f_{(i,j),\lambda},\vect f_{(k,\ell),\mu}\right\rangle =0$. On the other hand, if $(i,j)=(k,\ell)$, $\lambda\neq i$ and $\mu \ne k$, then it follows from the above equations that $B_{(i,j),\lambda}^T B_{(i,j),\mu}$ is a diagonal matrix, namely
\[
 B_{(i,j),\lambda}^T B_{(i,j),\mu} = \operatorname{diag}( 1, \langle \tensor{A}_{(i,j)}^\lambda, \tensor{A}_{(i,j)}^\mu \rangle, 1, \ldots, 1 ).
\]
Its determinant is then $\langle \tensor{A}_{(i,j)}^\lambda, \tensor{A}_{(i,j)}^\mu \rangle = \delta_{\lambda\mu} \| \vect{x}_\lambda \|^2$. Therefore,
\begin{align} \label{eqn_fijlinner}
 \langle \vect{f}_{(i,j)}^\lambda, \vect{f}_{(k,\ell)}^\mu \rangle = \delta_{ik} \delta_{j\ell} \delta_{\lambda\mu} \| \vect{x}_\lambda \|^2.
\end{align}
Finally, we can compute $\langle \vect{f}_{(i,j)}, \vect{f}_{(k,\ell)} \rangle$. From \cref{eqn_fij_def},
\[
\langle \vect{f}_{(i,j)}, \vect{f}_{(k,\ell)} \rangle
= s_{(i,j)} s_{(k,\ell)} \left\langle \sum_{1 \le \lambda \ne i \le d} \vect{f}_{(i,j)}^\lambda, \sum_{1 \le \mu \ne k \le d} \vect{f}_{(k,\ell)}^\mu \right\rangle
= s_{(i,j)} s_{(k,\ell)} \sum_{1 \le \lambda \ne i \le d} \delta_{ik}\delta_{j\ell} \| \vect{x}_\lambda \|^2,
\]
which is zero unless $(i,j) = (k,\ell)$. For $(i,j) = (k,\ell)$, we find
\[
 \| \vect{f}_{(i,j)} \|^2 = s_{(i,j)}^2 \sum_{1 \le \lambda \ne i \le d} \| \vect{x}_\lambda \|^2 = \sum_{1 \le \lambda \ne i \le d} \| \vect{x}_\lambda \|^2,
\]
proving the result.
\end{proof}
\bibliographystyle{amsplain}
\bibliography{literature}
\end{document}